\documentclass[twoside,12pt,reqno]{amsart}
\usepackage{bbm, tikz, tikz-cd}
\usetikzlibrary{positioning, arrows}
\usepackage{amssymb}
\usepackage{fullpage}
\usepackage[margin=2cm]{geometry}
\usepackage{bm, stmaryrd}
\usepackage{extarrows}

\usepackage{graphicx}
\input{xy}
\xyoption{all}
\xyoption{curve}
\usepackage{color,soul}

\newcommand{\myrightleftarrows}[1]{\mathrel{\substack{\xrightarrow{#1} \\[-.9ex] \xleftarrow{#1}}}}
\usepackage{stmaryrd}

\newcommand{\circleover}[2][0.15ex]{%
  \leavevmode
  \vbox{\offinterlineskip
    \ialign{##\cr
      \hidewidth$\scriptscriptstyle\circ$\hidewidth\cr
      \noalign{\kern#1}
      $#2$\cr
    }%
  }%
}
\newcommand{\oM}{\circleover{M}}

\makeatletter

\newtheorem{Theorem}{Theorem}[section]
\newtheorem*{Theorem*}{Theorem}
\newtheorem{Lemma}[Theorem]{Lemma}
\newtheorem{Proposition}[Theorem]{Proposition}
\newtheorem{Corollary}[Theorem]{Corollary}

\theoremstyle{remark}
\newtheorem{Remark}[Theorem]{Remark}

\numberwithin{equation}{section}

\usepackage[pdftex,colorlinks,backref,pagebackref,hypertexnames=false, linkcolor=blue,urlcolor=blue,citecolor=blue]{hyperref}

\newcommand{\arxiv}[1]{{\tt arXiv:#1}}

\def\iso{\cong}

\def\onto{\twoheadrightarrow}
\def\isoto{\overset{\sim}{\longrightarrow}}

\def\epsilon{\varepsilon}

\def\k{{\mathbbm k}}

\def\R{\mathbb{R}}
\def\G{\mathbb{G}}

\def\rn{{\mathrm n}}
\def\rs{{\mathrm s}}

\def\lmod{\!\operatorname{-mod}}

\def\Lie{\operatorname{Lie}}

\def\ad{\operatorname{ad}}
\def\Ad{\operatorname{Ad}}
\def\Ann{\operatorname{Ann}}

\def\coker{\operatorname{coker}}

\def\Der{\operatorname{Der}}

\def\Spec{\operatorname{Spec}}

\def\Id{\operatorname{Id}}
\def\Ind{\operatorname{Ind}}

\def\End{\operatorname{End}}

\def\Hom{\operatorname{Hom}}

\def\Aut{\operatorname{Aut}}

\def\op{{\operatorname{op}}}

\def\im{\operatorname{im}}

\def\rank{\operatorname{rank}}

\def\Supp{\operatorname{Supp}}
\def\QCoh{\operatorname{QCoh}}
\def\Coh{\operatorname{Coh}}

\def\o{\overline}

\def\C{{\mathbb C}}
\def\F{{\mathbb F}}

\def\O{{\mathbb O}}
\def\cO{{\overline{\mathbb O}}}

\def\one{{\mathbbm 1}}
\def\G{{\mathbb G}}

\def\GL{\mathrm{GL}}

\def\HC{\mathrm{HC}}
\def\soc{\operatorname{soc}}

\def\b{\mathfrak b}

\def\g{{\mathfrak g}}

\def\m{\mathfrak m}
\def\n{\mathfrak n}
\def\p{\mathfrak p}

\def\t{\mathfrak t}
\def\u{\mathfrak{u}}

\def\a{\mathfrak{a}}

\def\kk{\mathfrak{k}}
\def\a{\mathfrak{a}}
\def\b{\mathfrak{b}}

\def\cA{\mathcal{A}}

\def\cC{\mathcal{C}}
\def\cF{\mathcal{F}}
\def\cG{\mathcal{G}}

\def\cK{\mathcal{K}}
\def\U{\mathcal{U}}

\def\H{\mathcal{H}}
\def\Oc{\mathcal{O}}

\newcommand{\A}{\mathcal{A}}

\newcommand{\AS}{\operatorname{AS}}
\newcommand{\fg}{\operatorname{fg}}
\newcommand{\fd}{\operatorname{fd}}
\renewcommand{\rs}{{\rm s}}
\renewcommand{\rn}{{\rm n}}

\title{\boldmath Simple Harish-Chandra bimodules in positive characteristics}

\author{Lewis Groves, Lewis Topley and Matthew Westaway}

\address{Department of Mathematical Sciences, University of Bath, Claverton Down, Bath, BA2 7AY}
\email{ldg44@bath.ac.uk}
\email{lt803@bath.ac.uk}
\email{mpwestaway@gmail.com}


\begin{document}

    \begin{abstract}
    Let $G$ be a standard reductive group over an algebraically closed field of positive characteristic and let $\g = \operatorname{Lie}(G)$. The category of Harish-Chandra bimodules consists of $G$-equivariant $U(\g)$-bimodules such that the adjoint $\g$-action integrates to a $G$-action. In this paper we classify the simple Harish-Chandra bimodules and determine their central characters.
    \end{abstract}
	
	\maketitle

\section{Introduction}

Let $\k$ be an algebraically closed field and let $G$ be a reductive algebraic group over $\k$ with $\g=\Lie(G)$. A Harish-Chandra bimodule 
is a $U(\g)$-$U(\g)$-bimodule, finitely generated on both sides, with a compatible $G$-action. We refer the reader to \cite{Et} for a good introduction to the topic.

The theory of Harish-Chandra bimodules has a long, rich history over fields of characteristic zero. However, over fields of positive characteristic it remained largely unexplored until some important recent developments, see \cite{BR, BR2, Lo} for example.

In this paper, we classify simple Harish-Chandra bimodules over fields of very good positive characteristics. Before getting into the details of the classification, we provide a brief survey of the history of Harish-Chandra bimodules.

\subsection{Characteristic zero: Harish-Chandra, Bernstein--Gelfand and Soergel}

The story of Harish-Chandra bimodules begins in the 1950's, in the study of semisimple Lie groups. Let $G_\R$ be a real semisimple Lie group and let $K \subseteq G_\R$ be a maximal compact subgroup. Write $\g_\R=\Lie(G_\R)$ and $\kk=\Lie(K)$. In this generality, a Harish-Chandra module for $(\g_\R\otimes_\R\C,K)$ is defined to be a $K$-equivariant $\g_\R \otimes_\R \C$-module $M$ such that the $K$-action is algebraic (i.e. locally finite), $M$ is finitely generated over $U(\g)$, and the differential of the $K$-action coincides with the $\kk \subseteq \g$-action obtained via restriction. Using this framework, Harish-Chandra was able to show that inherently analytic questions about unitary representations of $G_\R$ on Hilbert spaces could be reduced to purely algebraic problems by considering the the largest algebraic $K$-submodule, now known as the algebraic core.

The version of Harish-Chandra bimodules we consider can be seen as a special case of this construction. Let $G_\C$ be a semisimple, simply-connected complex algebraic group and set $\g_\C = \Lie G_\C$. We may view $G_\C$ as a real Lie group, in which case $\g_\C \otimes_\R \C \cong \g_\C \oplus \g_\C$, and we may chose $K$ such that $\kk$ identifies with the skew-diagonal copy of $\g_\C$ inside $\g_\C \oplus \g_\C$. The category of Harish-Chandra modules studied in the present paper is the positive characteristic analogue of modules for the pair $(\g_\C \oplus \g_\C, K)$.

One reason to specialise to this case is that a plethora of tools  from the representation theory of complex simple Lie algebras become available. In their tour de force 1980 paper \cite{BG}, Bernstein and Gelfand classified irreducible Harish-Chandra bimodules via an ingenious scheme of argument. They noticed that tensoring with a Harish-Chandra bimodule induces an endofunctor of the Bernstein--Gelfand--Gelfand category $\Oc$ for $\g_\C$, and showed that the projective indecomposable Harish-Chandra bimodules correspond precisely to the projective endofunctors of $\Oc$. By first classifying the projective endofunctors, they were able to classify the projective covers of simple modules, and thus the simple modules themselves.

Recall that the centre $Z$ of $U(\g_\C)$ identifies with $S(\t_\C)^W$, where $\t_\C$ is a Cartan subalgebra of $\g_\C$ and $W$ is the Weyl group. Central characters of $U(\g_\C)$ thus correspond to points in $\t_\C^*/W$. It is often useful to consider Harish-Chandra bimodules according to their generalised central character; indeed, \cite{BG} shows that each simple Harish-Chandra bimodule has a central character. One further major achievement of \cite{BG} is to show an equivalence of categories between Harish-Chandra bimodules with a particular regular dominant central character and a certain subcategory of category $\Oc$.

In his famous paper \cite{So}, Soergel used a variation of this argument to construct a functor from the category of Harish-Chandra bimodules with trivial generalised central character on both sides to the category of $(Z,Z)$-bimodules, with a goal of gaining a combinatorial description of Harish-Chandra bimodules (this was then generalised by Stroppel in \cite{Str} to integral generalised central characters). One of the most lasting consequences of Soergel's scheme of argument was the construction of the Soergel bimodules, which are intimately connected to Kazhdan-Lustzig polynomials and the Kazhdan-Lusztig conjecture. In fact, Harish-Chandra bimodules can be more directly tied to the Kazhdan-Lusztig conjecture through means of the so-called principal series representations (see \cite[Theorem 6.7]{BG} for more details).

\subsection{Harish-Chandra bimodules in the modular setting}
 While Soergel's construction of Soergel bimodules post-dated the proof of the Kazhdan-Lusztig conjectures by around a decade, their combinatorial nature made them natural objects of study in relation to Lusztig's conjecture, which posited a relationship between the Kazhdan-Lusztig polynomials and character formulae for simple modules of semisimple algebraic groups $G_\k$ over fields $\k$ of positive characteristic. Although Lusztig's conjecture for $G_\k\lmod$ holds true for large primes, it was disproved by Williamson in \cite{Wi}, using ideas related to Soergel bimodules. In \cite{RW} Riche and Williamson posed a new conjecture relating the character formulae instead to the so-called $p$-Kazhdan-Lusztig polynomials. These can be defined from the diagrammatic Hecke category, which in characteristic zero coincides with the category of Soergel bimodules.

 Riche and Williamson showed that their conjectural character formula in \cite{RW} would follow from a different conjecture: the existence a suitably nice action of the so-called Hecke category on the principal block of $G_\k\lmod$ \cite[Conjecture 1.1]{RW}. While Riche-Williamson's original conjecture was subsequently proved by other means \cite{AMRW}, Bezrukavnikov and Riche proved in \cite{BR} that there is indeed the desired action of the Hecke category on the principal block. Their main tool to prove this was Harish-Chandra bimodules defined over fields of positive characteristic.

For further connections between Soergel bimodules and Harish-Chandra bimodules over fields of positive characteristic, see \cite{BR2, Lo}.

 \subsection{This paper}
 \label{ss:thispaper}

 Let us now turn to the setting of the current paper. We begin by fixing our assumptions. Let $\k$ be an algebraically closed field of characteristic $p > 0$ and let $G$ be a connected, reductive algebraic group over $\k$ with simply-connected derived subgroup. We assume that $p$ is good for the root system of $G$ and that $\g = \Lie G$ admits a non-degenerate $G$-invariant bilinear form. These assumptions are known as the standard hypotheses (see Section~\ref{ss:reductivegroupsandLiealgebras}) and we call $G$ a {\it standard reductive group}.

 For the remainder of the introduction, a {\it Harish-Chandra bimodule} is a $U(\g)$-$U(\g)$-bimodule equipped with a rational $G$-action satisfying the following three conditions.
\begin{enumerate}
    \item[(HC1)] The action map $U(\g) \otimes M \otimes U(\g) \to M$ is $G$-equivariant.
    \item[(HC2)] The differential of $G \to \End(M)$ coincides with $\ad_M : \g\to \End M$ given by $\ad_M(x)(m) = xm - mx$ for $x\in \g$ and $m\in M$.
    \item[(HC3)] $M$ is finitely generated as a left $U(\g)$-module (equivalently, as a right $U(\g)$-module).
\end{enumerate}
We denote by $\HC^G U(\g)$ the category of Harish-Chandra (HC) bimodules.

Over fields of positive characteristic, the universal enveloping algebra $U(\g)$ admits a central subalgebra $Z_p(\g)$ over which it is free of finite rank. This is called the $p$-centre of $U(\g)$, and it is ($G$-equivariantly) isomorphic to $\k[\g^*]^p$. Each $M\in\HC^G U(\g)$ is thus a $(Z_p(\g),Z_p(\g))$-bimodule; however, condition (2) above means that the left and right action must coincide.

We define the $p$-support of $M$ to be its (scheme-theoretic) support when viewed as a left (or, equivalently, right) $Z_p(\g)$-module. This must be a $G$-stable, closed subscheme of the Frobenius twist $(\g^*)^{(1)}$ of the coadjoint module, and given a subscheme $X$ with these properties we define by $\HC^G_X U(\g)$ the full subcategory of $\HC^G U(\g)$ consisting of HC bimodules with $p$-support inside $X$.

It is not hard to see that the $p$-support of a simple HC bimodule is a closed $G$-orbit, and thus a $G$-orbit consisting of semisimple elements in $\g^*$ (Lemma~\ref{L:suportsemisimple}). Classifying simple HC bimodules thus reduces to classifying simple objects in $\HC^G_\O U(\g)$ as $\O$ runs over the semisimple coadjoint $G$-orbits.

\subsection{Classification of simple objects in $\HC^G U(\g)$}
\label{ss:sketchclassificationintro}

By the above, 
it suffices to classify simple objects in 
$\HC_{\O}^G U(\g)$ for every 
semisimple orbit $\O$. Fix a maximal torus $T$ of $G$, let $\t=\Lie(T)$ and let $W$ be the Weyl group of $G$. Recall that the semisimple coadjoint $G$-orbits are indexed by $\t^*/W_\bullet$, the $W$-orbits of $\t^*$ under the dot-action.

Our starting point in the classification is to note that $\HC_\O^G U(\g)$ is equivalent to the category of $G$-equivariant, finitely generated $U_\O(\g)$-modules, denoted $U_\O(\g)\lmod^G_{\fg}$. Here $U_\O(\g) := U(\g) / I_\O U(\g)$ and $I_\O \subseteq Z_p(\g)$ is the defining ideal of $\O^{(1)} \subseteq (\g^*)^{(1)}$.

In Section~\ref{ss:reductionsoverSteinbergfibres} we prove a block decomposition theorem for $U_\O(\g)$ which implies that all indecomposable objects in $\HC_\O^G U(\g)$ admit (left and right) generalised central characters for $U(\g)^G \subseteq Z(\g)$.

The algebra $U_\O(\g)$ can be localised to give a coherent sheaf of $\Oc_{\O^{(1)}}$-algebras which we denote $\U_\O(\g)$. Since $\O^{(1)}$ is affine we have an equivalence
\begin{eqnarray}
\label{eq:intro1}
U_\O(\g)\lmod^G \isoto \U_\O(\g)\lmod^G,    
\end{eqnarray}
where the right hand category consists of sheaves of $G$-equivariant $\Oc_\O$-coherent $\U_\O(\g)$-modules.

Let $\chi \in \O$ be such that $T$ is a maximal torus in $G^\chi$, let $F : \O \to \O^{(1)}$ be the Frobenius morphism and let $G^{F(\chi)}$ be the (scheme-theoretic) stabiliser of $F(\chi) \in \O^{(1)}$. We note that $G^{F(\chi)}$ is a non-reduced group scheme whenever $\chi \ne 0$. More precisely, it is a Frobenius neighbourhood in $G$ of the stabiliser $G^\chi$, first studied in \cite[\textsection 2.4]{ATW}.

One of the key steps in our proof appears in Section~\ref{S:descent}, where we apply descent for $G$-equivariant coherent sheaves. Descent implies that specialising to $F(\chi)\in \O^{(1)}$ induces an equivalence of monoidal categories $\Coh^G(\O^{(1)}) \to G^{F(\chi)}\lmod$. Since $\U_\O(\g)$ is an algebra object in $\Coh^G(\O)$ we deduce an equivalence
\begin{eqnarray}
\label{eq:intro2}
\U_\O(\g)\lmod^G \isoto U_\chi(\g)\lmod^{G^{F(\chi)}},    
\end{eqnarray}
where $U_\chi(\g)$ denotes the reduced enveloping algebra.

Next, in Section~\ref{ss:equivariantreductiontonilpotent}, we upgrade a category equivalence of Friedlander--Parshall \cite[Theorem~3.2]{FP} to the equivariant setting. More precisely, we show that the study of $G^{F(\chi)}$-equivariant $U_\chi(\g)$-modules are equivalent to restricted HC bimodules
\begin{eqnarray}
\label{eq:intro3}
U_\chi(\g)\lmod^{G^{F(\chi)}} \isoto \HC_0^{G^\chi} U(\g^\chi). 
\end{eqnarray}

Now $G^\chi$ is a Levi subalgebra of $G$ and we have reduced to the consideration of restricted HC $U(\g^\chi)$-bimodules. In this setting, we are able to classify simple modules as follows. 

Fix a Borel subgroup $B$ of $G^\chi$ containing $T$, and let $X(T)$ be the character group of $T$. Let $\Phi^\chi$ be the root system of $G^\chi$ with respect to $T$, and let $\Pi^\chi\subseteq \Phi^\chi$ be the set of simple roots in $\Phi^\chi$ corresponding to $B$. We then write $X^\chi(T)_{+}$ for the set of characters of $T$ which are dominant with respect to $\Pi^\chi$. For $\gamma\in X^\chi(T)_{+}$ we denote by $L^\chi(\gamma)$ the simple $G$-module with highest weight $\gamma$ (see \cite[II.2.4]{JanRAGS}).

For $\lambda\in X(T)/pX(T)$ we denote by $L_0^\chi(\lambda)$ the simple restricted highest weight $U(\g^\chi)$-module with highest weight $\lambda$ (see \cite[\textsection 10]{JaLA}). 
 Write $L_0^\chi(\lambda) \boxtimes L_0^\chi(\mu) \in U_0(\g^\chi) \otimes U_0(\g^\chi)\lmod$ for the exterior tensor product of simple $U(\g^\chi)$-modules, which can be extended to a $G^\chi$-module (see \eqref{eq:Frobeniuskernelandrestrictedenvelopingalgebra} and Lemma~\ref{L:liftingtoG}) and then viewed as a $G^\chi$-equivariant $U_0(\g^\chi)$-bimodule.

We also denote by $L^\chi(\gamma)^F$ denotes the pullback under the Frobenius morphism of the $G^{(1)}$-module $L^\chi(\gamma)^{(1)}$ to a $G$-module, and equip it with trivial $U_0(\g^\chi) \otimes U_0(\g^\chi)$-module structure.

For each $W \chi \in \t^* / W$ fix a representative $\chi$, and for each such $\chi$ choose a one dimensional $U_\chi(\g)\otimes U_{-\chi}(\g)$-module $E_\chi$ (see Corollary~\ref{thm: FP nilpotent reduction}). Suppose that $E_\chi$ has left and right Harish-Chandra central character $(\lambda_\chi, -\lambda_\chi)$. As we explain in Sections~\ref{ss:HCbimods} and \ref{ss:psuppHCbimods}, this module can be viewed as an object in $\HC^{G^\chi} U(\g^\chi)$.

It follows from all of the above discussion that we have an equivalence of categories $\HC_\O^G U(\g) \to \HC_0^{G^\chi} U(\g^\chi)$ and, to formulate our main theorem, we denote the quasi-inverse functor by 
\begin{eqnarray}
\label{eq:pinkapples}
    \Gamma \colon \HC_0^{G^\chi} U(\g^\chi) \isoto \HC_\O^{G} U(\g).
\end{eqnarray}
In Section~\ref{ss:centralcharactersunderparabolic} we describe the functor \eqref{eq:pinkapples} in more detail, in particular, see \eqref{eq:popcandy} and \eqref{eq:HCrestrictedtoreduced}.

\begin{Theorem}\label{T:MainThm}
    \begin{enumerate}
    \setlength{\itemsep}{6pt}
        \item If $M \in \HC^G U(\g)$ is simple then $M$ is supported on a semisimple orbit $\O$.
        \item If $\chi \in \O$ is semisimple then a complete set of non-isomorphic simple objects in $\HC_\O^G U(\g)$ is given by
        $$\Gamma \Big(\big((L_0^\chi(\mu_1) \boxtimes L_0^\chi(\mu_2)\big) \otimes L^\chi(\gamma)^F\Big),$$
        where $(\mu_1, \mu_2, \gamma)$ vary over $\mu_1, \mu_2 \in X(T)/pX(T)$ and $\gamma \in X^\chi(T)_+$.
        \item The object corresponding to the quadruple $(\chi,\mu_1,\mu_2,\gamma)$ has left Harish-Chandra central character $W_\bullet(\mu_1 + \lambda_\chi)$ and right Harish-Chandra central character $W_\bullet (\mu_2 - \lambda_\chi)$.
    \end{enumerate}

\end{Theorem}

The reader should note the similarity to Steinberg's tensor product theorem, and more directly to results of \cite[\textsection 11]{Vu}, which classifies simple HC bimodules for a certain quotient of the small quantum group over $\C$ (these correspond to $\chi = 0$ in our setting). In fact, the case $\chi = 0$ of our theorem can be extracted from \cite[Theorem~5.5]{Ho}; we remark that the latter paper focuses on finite dimensional modules.

\subsection{Generalisations and future work}

The methods of the current paper are slightly more general than was needed to outline the classification of simple HC bimodules, and we would like to highlight some interesting generalities which we discovered in the course of our investigation.

If $\O \subseteq \g^*$ is an arbitrary coadjoint orbit then one may always localise to $\overline\O$, giving an equivalence similar to \eqref{eq:intro1}, and specialise at $\chi\in\O$ to give a functor
\begin{eqnarray}
\label{eq:generalequiv}
    U_{\overline{\O}}(\g)\lmod^G \to U_{\chi}(\g)\lmod^{G^{F(\chi)}}.
\end{eqnarray}
In Propositions~\ref{P:MainEquiv} and \ref{P:generalequivalencebydescent} we show that \eqref{eq:generalequiv} admits a fully faithful right adjoint, hence induces an equivalence from the quotient category
\begin{eqnarray}
    U_{\overline{\O}}(\g)\lmod^G / \mathcal{K} \isoto U_\chi(\g)\lmod^{G^{F(\chi)}},
\end{eqnarray}
where $\mathcal{K}$ is the subcategory of $U_{\overline{\O}}(\g)\lmod^G$ generated by $U_{\delta \overline\O}(\g)\lmod^G$, and $\delta \overline\O := \overline\O \setminus \O$ denotes the boundary, i.e. $\mathcal{K}$ is the smallest extension-closed, full abelian category containing $U_{\delta \overline\O}(\g)\lmod^G$ (see Lemma~\ref{L:tiptop}). Combining this equivalence with \eqref{eq:intro2}, and Premet's equivalence (see \cite[Lemma~2.2(ii)]{Pr} for example) we obtain a functor to equivariant modules over the reduced finite $W$-algebra, which is reminiscent of a theorem of Losev \cite[Theorem~1.3.1(5)]{LoFDRWA}.

In future work we will describe the relationship between bimodules over finite $W$-algebras and HC bimodules in positive characteristics, extending some of the results of Losev \cite{LoFDRWA} and Ginzburg \cite{Gi} to the modular setting. The category $\HC^G U(\g)$ appears to be a rich source of interesting problems, some elementary and some deep. 

Another result worth drawing attention to is our Theorem~\ref{T:simplesFrobeniusneighbourhood}, which classifies the simple $K$-modules where $K$ is the Frobenius neighbourhood of any algebraic subgroup scheme of a standard reductive group $G$. Our Theorem~\ref{T:MainThm} is a special case of this result, after identifying $\HC_0^G U(\g)$ with modules over an appropriate group scheme (see Lemma~\ref{L:monoidalequivalence}). In fact Lemma~\ref{L:HCequivalencewithequivariance} and Theorem~\ref{T:simplesFrobeniusneighbourhood} lead to a classification of simple restricted modules over any Harish-Chandra pair $(\g, K)$ defined over $\k$, in terms of the simple restricted $\g$-modules and simple $K$-modules. The methods of this paper will extend to classify the simple modules over a Harish-Chandra pair $(\g, K)$ when $(G,K)$ is any symmetric pair, and this will be pursued in future work.

Finally, we would like to flag our Corollary~\ref{C:semisimplesurpiseequivariance}, which is a by-product of our proof of \eqref{eq:intro3}. It states that when $\chi \in \g^*$ is semisimple, we have an equivalence of categories
\begin{eqnarray}\label{eq:surprisingcorollary}
    U_\chi(\g)\lmod^{G^{F(\chi)}} \isoto U_\chi(\g)\lmod^{G^\chi}.
\end{eqnarray}
It is surprising, in part, because \eqref{eq:surprisingcorollary} is not the forgetful functor associated with the restriction map $G^{F(\chi)}\lmod \to G^\chi\lmod$, see Remark~\ref{R:curiousequiv} for more detail.

Finally we observe that the parameterisation in Theorem~\ref{T:MainThm} depends on two choices which must be made (independently) for every element of $\t^*/W$. We expect that a choice-free parameterisation is possible, similar to \cite{BG}, and we hope to explore this in future.

\subsection*{Acknowledgements}
The second and third author are grateful for support from the UKRI Future Leaders fellowship, grant numbers MR/S032657/1, MR/S032657/2, MR/S032657/3, MR/Z000394/1. The authors are also grateful for useful discussions with Simon Riche and Trung Vu during this project. We offer special thanks to Jay Taylor who suggested the proof of Lemma~\ref{L:liftingtoG}.

\tableofcontents

\section{Preliminaries}

In this section we establish some definitions, notation and results which will be used throughout the paper. The majority of this is well-known, though in Subsection~\ref{ss:reductionsoverSteinbergfibres} we prove some new results about block decompositions of certain central quotients of universal enveloping algebras.

\subsection{Conventions}

Throughout $\k$ is an algebraically closed field of positive characteristic $p > 0$, and unadorned tensor products are taken over $\k$. All algebras, vector spaces and schemes are defined over $\k$.

If $A$ is an algebra then $A\lmod$ denotes all left $A$-modules whilst $A\lmod_{\fg}$ denotes the category of finitely generated left $A$-modules and $A\lmod_{\fd}$ denotes the finite-dimensional modules. We use similar notation for rational representations of algebraic group schemes.

If $X$ is a scheme then $\Coh(X)$ denotes the category of coherent sheaves on $X$. We also write $X^{(1)}$ for the Frobenius twisted $\k$-scheme and $F_X : X \to X^{(1)}$ for the Frobenius morphism. 

 \subsection{Monoidal categories}

 In this subsection, we recall the notion of a monoidal category and an equivalence of monoidal categories. We then observe that such an equivalence preserves the structure of an algebra and a module over that algebra. The reader may consult \cite{EGNO} for more details. All functors and categories are assumed to be $\k$-linear, so we omit this from the notation.

 A monoidal category consists of a category $\cC$, a bifunctor $\otimes:\cC\times\cC\to \cC$, a monoidal unit $\one\in\cC$, and three natural isomorphisms $a$, $l$ and $r$, which are the associator, and a left and right unit isomorphism, respectively.
 
 Such a datum is required to satisfy the pentagon axiom and two unit axioms (see \cite[Definition 2.1.1]{EGNO}). A monoidal functor between two monoidal categories $(\cC,\otimes,\one,a,l,r)$ and $(\cC',\otimes',\one',a',l',r')$ is a triple $(F,J,\phi)$ consisting of the following data: 
\begin{enumerate}
    \item A functor $F:\cC\to\cC'$;
    \item A natural isomorphism $J:F(-)\otimes' F(-)\to F(-\otimes -)$; and 
    \item An isomorphism $\phi:\one'\to F(\one)$.
\end{enumerate}
This datum is required to satisfy three axioms: the monoidal structure axiom and two axioms related to the monoidal unit (see \cite[Definition 2.4.5]{EGNO}). An equivalence of monoidal categories is a monoidal functor which is an equivalence of the underlying categories.

With data fixed as above, we might refer to $\cC$ as the monoidal category, and suppress the other notation. An {\it algebra in $\cC$} consists of a triple $(A,m,u)$ where $A\in\cC$, $m\in\Hom_\cC(A\otimes A,A)$, and $u\in\Hom_\cC(\one,A)$, such that $A$, $m$ and $u$ satisfy the usual axioms for an algebra (i.e. associativity and the left and right unitary properties).

The following result is straightforward.
 \begin{Lemma}\label{L:Falgisalg}
     Let $\cC$ and $\cC'$ be monoidal categories and let $(F,J,\phi)$ be a monoidal functor $\cC\to\cC'$. Suppose $(A,m,u)$ is an algebra in $\cC$. Then $(F(A),F(m)\circ J_{A,A},F(u)\circ\phi)$ is an algebra in $\cC'$.
 \end{Lemma}

 Given an algebra $A$ in $\cC$, the category of {\it $A$-modules in $\cC$}  is denoted $A\lmod_\cC$, and defined as follows. Objects consist of pairs $(M,\rho)$, where $M$ is an object in $\cC$ and $\rho\in\Hom_\cC(A\otimes M,M)$, such that the following diagrams commute
\begin{eqnarray}
\begin{array}{c}\xymatrix{
     A\otimes (A\otimes M)\ar@{->}[rrr]^{\Id\otimes\rho}  \ar@{->}[d]^{a_{A,A,M}} & & & A\otimes M \ar@{->}[d]_{\rho} & \one\otimes M \ar@{->}[rr]^{u\otimes \Id} \ar@{->}[rrd]^{l_{\one,M}} & &  A\otimes M \ar@{->}[d]^{\rho}
     \\
     (A\otimes A)\otimes M \ar@{->}[rr]^{m\otimes \Id} & & A\otimes M \ar@{->}[r]^{\rho} & M & & & M
    }
\end{array}
\end{eqnarray}
 A morphism $(M,\rho)\to (M',\rho')$ in $A\lmod_\cC$ is then  defined to be a morphism $\Phi:M\to M'$ in $\cC$ such that the following diagram commutes
\begin{eqnarray}
\begin{array}{c}\xymatrix{
    A\otimes M \ar@{->}[rr]^{\rho}  \ar@{->}[d]^{\Id\otimes\Phi} & & M \ar@{->}[d]_{\Phi}
     \\
     A\otimes M' \ar@{->}[rr]^{\rho'} & & M'
    }
\end{array}
  \end{eqnarray}

The following lemma is also straightforward.
 \begin{Lemma}\label{L:Fmodismod}
     Let $\cC$ and $\cC'$ be monoidal categories and let $(F,J,\phi)$ be a monoidal functor $\cC\to\cC'$. Suppose $A=(A,m,u)$ is an algebra in $\cC$, and that $(M,\rho)$ is an $A$-module in $\cC$. The following are true.
     \begin{enumerate}
         \item $(F(M),F(\rho)\circ J_{A,M})\in F(A)\lmod_{\cC'}$.
         \item The map $(M,\rho)\mapsto (F(M),F(\rho)\circ J_{A,M})$ defines a functor $A\lmod_{\cC}\to F(A)\lmod_{\cC'}$.
         \item If $F$ is an equivalence of categories, then $A\lmod_{\cC}\to F(A)\lmod_{\cC'}$ is also an equivalence of categories.
     \end{enumerate}
 \end{Lemma}
 
 For part (3) of this lemma, one needs to use the observation in \cite[Remark 2.4.10]{EGNO} that for a monoidal functor $F$ which is an equivalence of categories, any quasi-inverse $F^{-1}$ can be given the structure of a monoidal functor such that the natural isomorphisms $F\circ F^{-1}\isoto \Id$ and $\Id\isoto F^{-1} \circ F$ are isomorphisms of monoidal functors (see \cite[Definition 2.4.8]{EGNO}).
 
 Furthermore, we may additionally assume that $F^{-1}$ is a left and right adjoint to $F$. Then $F^{-1}$ induces a functor $F(A)\lmod_{\cC'}\to F^{-1} \circ F(A)\lmod_{\cC}$, and using the isomorphism $A\isoto F^{-1}\circ F(A)$ we obtain a functor $F(A)\lmod_{\cC'}\to A\lmod_{\cC}$. It is straightforward to check that this is a quasi-inverse to the functor $A\lmod_{\cC}\to F(A)\lmod_{\cC'}$.

\subsection{Restricted Lie algebras and reductive groups}
\label{ss:reductivegroupsandLiealgebras}

Recall that a Lie algebra $\g$ is called restricted if it is equipped with a $p$-mapping $x \mapsto x^{[p]}$ such that the map $\g \to U(\g)$ defined by $x\mapsto  x^p - x^{[p]}$ induces a homomorphism of algebras $S(\g^{(1)}) \to Z(\g)$, where $Z(\g)$ is the centre of $U(\g)$. By the PBW theorem, this map is always injective. The image is called the $p$-centre $Z_p(\g)$. By construction we have $\Spec Z_p(\g) = (\g^*)^{(1)}$. For $\chi\in\g^*$, the reduced enveloping algebra corresponding to $\chi$ is defined to be $U_\chi(\g):=U(\g)/I_\chi U(\g)$ where $I_\chi \unlhd Z_p(\g)$ is the ideal generated by $\{ x^p-x^{[p]}-\chi(x)^p\mid x\in\g\}$. More generally, if $X$ is a closed subscheme of $(\g^*)^{(1)}$ then we write $U_X(\g):=U(\g)/I_X U(\g)$, where $I_X$ is the defining ideal of $X$ in $Z_p(\g)$. The restricted enveloping algebra is $U_0(\g)$ and $U_0(\g)$-modules are called restricted $U(\g)$-modules.

If $A$ is an associative algebra then the Lie algebra $\Der A$ of derivations of $A$ is restricted, with $p$-mapping $\delta \mapsto \delta^p$. If $G \subseteq \Aut A$ is a group of automorphisms then the set of derivations that commute with $G$ is a restricted subalgebra of $\Der A$. Therefore, if $G$ is a group scheme then the Lie algebra $\g = \Lie G$ is restricted, and the $p$-mapping is $G$-equivariant. The latter implies that the identification $\Spec Z_p(\g) = (\g^*)^{(1)}$ is $G$-equivariant.

Let $G$ now be a reductive group (which we generally view as a group scheme), and assume the standard hypotheses from \cite[\textsection 6.3]{JaLA}. In particular:
\begin{itemize}
    \item[(H1)] the derived subgroup of $G$ is simply-connected;
    \item[(H2)] $p$ is good for the root system of $G$;
    \item[(H3)] the adjoint $G$-module is self-dual.
\end{itemize}
We call such a group {\it standard reductive}.

 Choose a maximal torus $T \subseteq G$ with Lie algebra $\t \subseteq \g$. Fixing a choice of positive roots $\Phi^+ \subseteq \Phi \subseteq X^*(T)$, where $X^*(T)$ is the character group of $T$, we obtain a triangular decomposition $\g = \n^- \oplus \t \oplus \n$, and Borel subalgebra $\b = \t \oplus \n$ corresponding to a Borel subgroup $B \subseteq G$. Write $\rho \subseteq \t^*$ for the sum of the fundamental weights.

 The Weyl group $W = N_G(T) / T$ acts on $\t^*$ naturally $(w, \lambda) \mapsto w\lambda$, and also acts via the $\rho$-shifted ``dot action'' $(w, \lambda) \mapsto w \bullet \lambda := w(\lambda + \rho) - \rho$.

 A Levi subgroup of $G$ is the Levi factor of a parabolic subgroup. When $G$ is standard reductive, every Levi subgroup is also connected and standard reductive. 

 Hypothesis (H3) implies that there exists a $G$-isomorphism $\g \to \g^*$ and using this we can define the notion of semisimple and nilpotent elements of $\g^*$, as well as the Jordan decomposition for arbitrary $\chi \in \g^*$. The following well-known fact will be crucial for our deductions, and it follows from \cite[Proposition~7.13]{JaNO}.
\begin{Lemma}
\label{L:Zariskiclosedsemisimple}
     The Zariski closed orbits in $\g^*$ are precisely the orbits of semisimple elements. Furthermore the centralisers in $G$ of semisimple elements are precisely the Levi subgroups. 
\end{Lemma}

\subsection{The Harish-Chandra centre}\label{ss:HCcentre}
A remarkable achievement of Harish-Chandra was his determination of the centre of the enveloping algebra of a complex semisimple Lie algebra, building on the work of Chevalley who described the symmetric invariants. His ideas continue to pervade over fields of positive characteristic, as we now explain.

  The invariants in $\k[\t^*] = S(\t)$ under the dot action of $W$ are denoted $S(\t)^{W_\bullet}$, whilst the corresponding quotient is written $\t^* \to \t^* / W_{\bullet}$. The Chevalley restriction theorem \cite[Proposition 7.12]{JanNO} states that the restriction map $\k[\g] \to \k[\t]$ and the identifications $\g\isoto\g^*$ and $\t\isoto \t^*$ induce an isomorphism
  \begin{eqnarray}
  \label{eq:Chevalleyrestriction}
       \t^* / W \isoto \g^* /\!/ G.
  \end{eqnarray}
  
  The Harish-Chandra (HC) homomorphism is an embedding $U(\g)^G \to U(\t)$ that induces a natural isomorphism 
  \begin{eqnarray}
  \label{eq:HCrestriction}
  U(\g)^G \cong \k[\t^*]^{W_\bullet}.    
  \end{eqnarray}
  If we consider the subspace $I = U(\g)\n^+ \cap \n^- U(\g)$ then Harish-Chandra observed that the $T$-invariants $I^T$ form an ideal in $U(\g)^T$, and the quotient induces $U(\g)^T / I^T \cong U(\t)$. The HC homomorphism is the composition $U(\g)^G \subseteq U(\g)^T \onto U(\t) = S(\t)$.

If $M \in \g\lmod$ then we say that $M$ admits a Harish-Chandra (HC) central character if the map $U(\g)^G \to \End(M)$ factors through a character (i.e. an algebra homomorphism) $\psi : U(\g)^G \to \k$. We identify the set of HC central characters with $\t^* / W_\bullet$ via \eqref{eq:HCrestriction}. If $M$ is a highest weight $\g$-module generated by a vector of weight $\lambda \in \t^*$, then the central character is $W_\bullet\lambda \in \t^*/W_\bullet$.

 Since $U(\g)$ is a rational $G$-module under the adjoint action and $G$ is connected, we have $U(\g)^G \subseteq U(\g)^\g = Z(\g)$.  It follows that the centre $Z(\g) \subseteq U(\g)$ contains $U(\g)^G$ and $Z_p(\g)$. Veldkamp proved that these two subalgebras generate the centre. Although his proof assumed that $G$ is simple and $p$ is larger than the Coxeter number, the theorem continues to hold under the standard hypotheses of Section~\ref{ss:reductivegroupsandLiealgebras}, see \cite[\textsection 9]{JaLA}.

In fact $Z(\g) = Z_p(\g) \otimes_{Z_p(\g)^G} U(\g)^G$, and this leads to a fibre product description
\begin{eqnarray}
\label{eq:centreUEA}
\Spec Z(\g) = (\g^*)^{(1)} \times_{(\t^* / W)^{(1)}} \t^* /W_\bullet
\end{eqnarray}
Here the map $(\g^*)^{(1)} \to (\g^*/\!/ G)^{(1)} = (\t^* / W)^{(1)}$ is the coadjoint quotient map followed by the identification \eqref{eq:Chevalleyrestriction}. The map 
\begin{eqnarray}
    \label{eq:AS}
    \AS : \t^* / W_\bullet \to (\t^* / W)^{(1)}
\end{eqnarray}
is referred to as the Artin-Schreier (AS) map, and is dual to the inclusion 
\begin{eqnarray}
\label{eq:ArtinSchreier}
 \k[(\t^*)^{(1)}]^W = \k[(\g^*)^{(1)}]^G = Z_p(\g)^G \subseteq U(\g)^G = \k[\t^*]^{W_\bullet}.   
\end{eqnarray}
The extremal equalities in \eqref{eq:ArtinSchreier} are obtained by identifying through the Chevalley restriction \eqref{eq:Chevalleyrestriction} and Harish-Chandra restriction theorem \eqref{eq:HCrestriction}, respectively.

In fact, it was observed in \cite[\textsection 2.3]{BMR} that $\AS$ can be described slightly more directly. Precisely, the AS map $S(\t) \to S(\t)^{(1)} = \k[(\t^*)^{(1)}]$ is determined by $t \mapsto t^p - t^{[p]}$ for $t \in \t$. It is not hard to check that this induces a homomorphism $S(\t^{(1)})^W \to S(\t)^{W_\bullet}$, and this corresponds to the map $Z_p(\g)^G \to U(\g)^G$ after appropriate identifications.

\subsection{Generalised central characters for reductions over Steinberg fibres}
\label{ss:reductionsoverSteinbergfibres}

The coadjoint quotient map for $G \curvearrowright \g^*$ is often called the Steinberg map \cite[7.2]{JaNO}. Here we study uniform properties of modules supported on the fibres of the Steinberg map.

Let 
\begin{eqnarray}
    \pi : (\g^*)^{(1)} \to (\g^*)^{(1)} /\!/ G = (\t^* / W)^{(1)}
\end{eqnarray}
denote the Frobenius twisted Steinberg map.
Throughout this section we fix some $[\lambda] \in (\t^* / W)^{(1)}$ and write $X$ for the fibre $\pi^{-1} [\lambda]$. Recall that we write $U_X (\g)$ for the quotient $U(\g) / I_X U(\g)$ where $I_X$ is the defining ideal of $X$ in $Z_p(\g)$. Given $\psi\in\Hom(U(\g)^G,\k)$, we furthermore write $U_X^{(\psi)}(\g)$ for the quotient $U_X (\g) / (\ker \psi)^d U_X (\g)$ where $d = p^{\dim \g}$.

The first lemma we need is a description of the closed points in $\AS^{-1}[\lambda]$, where $\AS$ is as defined in Subsection~\ref{ss:HCcentre}. To state it, for each $\eta\in(\g^*)^{(1)}$ we define $Z_\eta(\g)$ to be the image of $Z(\g)$ under the natural map $U(\g)\to U_\eta(\g)$. Note that this need not coincide with the centre of $U_\eta(\g)$; see, for example, \cite[Example 3.19]{Brown Gordon}. Note also that the set of closed points of $\AS^{-1}[\lambda]$ is a finite set.

\begin{Lemma} \label{L: artin schreier preimage}
    For all $\eta \in X$, we have $\AS^{-1}[\lambda] = \Spec Z_\eta(\g)$ as schemes. These are precisely the Harish-Chanda central characters associated to $U_X(\g)$-modules.
\end{Lemma}

\begin{proof}
    Since the Harish-Chandra centre $U(\g)^G$ surjects onto $Z_\eta(\g)$, we have that $Z_\eta(\g) = \im(U(\g)^G \to U_\eta(\g))$.
    Therefore a maximal ideal of $Z_\eta(\g)$ corresponds to a maximal ideal of $U(\g)^G$ whose image remains proper in $Z_\eta(\g)$.
    Thus
    \begin{equation}
        \Spec Z_\eta(\g) = \{ \psi \in \Hom(U(\g)^G, \mathbbm{k}) \mid \ker(U(\g)^G \to U_\eta(\g)) \subseteq \ker \psi \}.
    \end{equation}
   Note also that $\ker(U(\g)^G \to U_\eta(\g)) = I_\eta U(\g) \cap U(\g)^G$.

    The Artin-Schreier map  is the morphism associated to the inclusion $Z_p(\g)^G \subseteq U(\g)^G$, see \eqref{eq:ArtinSchreier}.
    Identify $Z_p(\g)^G = \k[(\t^*)^{(1)}/W]$ via \eqref{eq:Chevalleyrestriction}.
    If $J_\lambda$ denotes the defining ideal of $[\lambda]$ in $Z_p(\g)^G$ then $\AS^{-1}[\lambda]$ corresponds to the maximal ideals of $U(\g)^G$ containing $J_\lambda U(\g)^G$.
    Therefore
    \begin{equation}
        \AS^{-1}[\lambda] = \{ \psi \in \Hom(U(\g)^G, \mathbbm{k}) \mid J_\lambda U(\g)^G \subseteq \ker \psi \}.
    \end{equation}

    To prove the current lemma, it is enough to show that $J_\lambda U(\g)^G = I_\eta U(\g) \cap U(\g)^G$.
    Since $\eta \in X$, we have that $J_\lambda \subseteq I_\eta$ and so $J_\lambda = I_\eta^G$.
    Hence it is clear that $I_\eta^G U(\g)^G \subseteq I_\eta U(\g) \cap U(\g)^G$. Equivalently we have a surjection of algebras
    \begin{eqnarray}
    \label{eq:alghom}
        U(\g)^G / I_\eta^G U(\g)^G \onto U(\g)^G / I_\eta U(\g) \cap U(\g)^G.
    \end{eqnarray}

    By \cite[Theorem 8.2]{Special Slices} we have that $\im(U(\g)^G \to U_\eta(\g))$ has dimension $p^{\rank \g}$.
    Further, $J_\lambda = I_\eta^G$ is a maximal ideal $Z_p(\g)^G$ and from \cite[Theorem 3.5(2)]{Brown Gordon} we know that $U(\g)^G$ is a free $Z_p(\g)^G$-module of rank $p^{\rank \g}$. The latter implies that the left hand side of \eqref{eq:alghom} also has dimension $p^{\rank\g}.$
    This implies that \eqref{eq:alghom} is an isomorphism. This implies $J_\lambda U(\g)^G = I_\eta U(\g) \cap U(\g)^G$ and the proof is complete.
\end{proof}

\begin{Theorem} \label{T:blockdecomp}
    The diagonal map
    \begin{equation} \label{eq: block decomp phi}
        \phi \colon U_X(\g) \to \bigoplus_{\psi \in \AS^{-1}[\lambda]} U_X^{(\psi)} (\g)
    \end{equation}
    is an isomorphism.
\end{Theorem}

\begin{proof}
   Write $U_\eta^{(\psi)}(\g)$ for the quotient $U_\eta(\g)/\ker(\psi)^dU_\eta(\g)$. We first prove that the specialised map
\begin{equation} \label{eq: phi eta}
    \phi_\eta \colon U_\eta(\g) \to \bigoplus_{\psi \in \AS^{-1}[\lambda]} U_\eta^{(\psi)}(\g) 
\end{equation}
is an isomorphism for any $\eta\in X$.
By a result of M\"{u}ller \cite[Theorem 7]{Muller} (also see \cite[2.10]{Brown Gordon} for the theorem applied directly to our context) we have the decomposition into blocks
\begin{equation}
    U_\eta(\g) = \bigoplus_{i=1}^k U_\eta(\g) e_i
\end{equation}
where $e_1, \ldots, e_k$ are the primitive central idempotents in $Z_\eta(\g)$.
For each $e_i$ we let $I_i$ be the unique maximal ideal in the local ring $Z_\eta(\g) e_i$, see \cite[Theorem 8.7]{AM}.

We then define
\begin{equation}
    J_i = I_i \oplus \bigoplus_{j \neq i} Z_\eta(\g) e_j.
\end{equation}
Since $I_i$ is nilpotent, with nilpotency index bounded above by $d = \dim U_\eta(\g)$, we then have
\begin{equation}
    J_i^d = \bigoplus_{j \neq i} Z_\eta(\g) e_j
\end{equation}
so that $U_\eta(\g) / J_i^d U_\eta(\g) \cong U_\eta(\g) e_i$.
Thus by applying Lemma \ref{L: artin schreier preimage} we see that (\ref{eq: phi eta}) is an isomorphism.

{\it Claim 1:}    The map $\phi$ given in (\ref{eq: block decomp phi}) is surjective.

    Note that $\ker(Z_p(\g) \to U_X(\g)) = Z_p(\g) \cap I_X U(\g)$ and the latter is equal to $I_X$ since $U(\g)$ is free (hence faithfully flat) over the $p$-centre.
    It follows that the domain and codomain of $\phi$ are both $\k[X^{(1)}]$-modules that are finitely generated. Let $C$ denote the cokernel of $\phi$ in $\k[X^{(1)}]\lmod_{\fg}$.    

    Assume $C \neq 0$. Then there exists a maximal submodule $M \subset C$ such that $C / M \cong \k[X^{(1)}] / I_\eta$ for some $\eta \in X$. In particular $C / I_\eta C \ne 0$
    As specialisation is right exact we have that $C / I_\eta C = \coker(\phi_\eta)$ and the latter is zero since $\phi_\eta$ is an isomorphism. This contradiction confirms that $C = 0$ and so $\phi$ is surjective.

{\it Claim 2:} The map $\phi$ given in (\ref{eq: block decomp phi}) is injective.

    Note that $U_X(\g)$ is a free (hence projective) $\k[X^{(1)}]$-module of rank $p^{\dim \g}$.  Since $\phi_\eta$ is an isomorphism, it follows that $\bigoplus_\psi U_X^{(\psi)}(\g)$ is a projective $\mathbbm{k}[X^{(1)}]$-module of the same rank. It follows that $\ker \phi$ is a projective module of rank zero and hence must be zero.
\end{proof}

\subsection{Frobenius neighbourhoods of subgroup schemes}
 \label{ss:stabilisersoftwistedcharacters}

If $G$ is a group scheme then so is $G^{(1)}$ and the Frobenius morphism $F_G$ is a homomorphism of group schemes \cite[I.9]{JanRAGS}.

Let $G$ be an algebraic group scheme, not necessarily reductive. The first Frobenius kernel of $G$ is defined to be the kernel of $F_G$, which is an infinitesimal, normal subgroup scheme of $G$. Its structure and representation theory are described in detail in \cite[II.2 \& II.3]{JanRAGS}. The most important feature for our purposes is the following equivalence of categories
\begin{eqnarray}
    \label{eq:Frobeniuskernelandrestrictedenvelopingalgebra}
    G_1\lmod \isoto U_0(\g)\lmod.
\end{eqnarray}
This can be proven by combining the remarks preceding I.8.4(2) with I.8.6(2) and I.9.6(4) in \cite{JanRAGS}.

We regard $G_1$ as a Frobenius neighbourhood of the trivial subgroup scheme $\{1\} \subseteq G$. At several points in our paper we will need to consider the structure and representations of Frobenius neighbourhoods of more general subgroup schemes, and we prepare the groundwork here.

Let $H \subseteq G$ be a closed subgroup scheme and define $K:= F_G^{-1}(H^{(1)})$. We have the following fibre product description
\begin{eqnarray}\label{e:defineK}
    K := G \times_{G^{(1)}} H^{(1)}.
\end{eqnarray}
Consider the two subgroups of $K$
\begin{eqnarray}
    H \times_{G^{(1)}} H^{(1)} & \cong & H\\
    G \times_{G^{(1)}} \{1\} & \cong & G_1.
\end{eqnarray}
It is not hard to see that $K$ is the closed subgroup scheme of $G$ generated by $H$ and $G_1$, and that multiplication in $G$ defines a surjective homomorphism of group schemes
\begin{eqnarray}
\label{eq:semidirectsurjection}
    G_1 \rtimes H \onto K.
\end{eqnarray}
\begin{Lemma}\label{L:semidir}
    The homomorphism \eqref{eq:semidirectsurjection} induces an isomorphism
    \begin{eqnarray}
        \label{eq:FITgoupschemes}
        (G_1 \rtimes H) / H_1 \isoto K
    \end{eqnarray}
    
    where $H_1(A) = \{(g,g^{-1}) \in G_1(A) \rtimes H(A)\}$ for any commutative $\k$-algebra $A$.
\end{Lemma}
\begin{proof}
    First of all, note that $H_1 \subseteq G_1 \rtimes H$ defined in the statement of the lemma is contained in the kernel of \eqref{eq:semidirectsurjection}. Furthermore, the map $K = G \times_{G^{(1)}} H^{(1)} \to G$ is injective so if $A$ is a $\k$-algebra, $(g,1) \in G_1(A) \subseteq K(A)$ and $(h, F(h)) \in H(A) \subseteq K(A)$ then whenever $(g,1) , (h, F(h))$ maps to the identity element under \eqref{eq:semidirectsurjection}, we must have $gh = 1$, i.e. $g = h^{-1}$. This shows that $H_1$ is precisely the kernel of \eqref{eq:semidirectsurjection}.

    Now we can apply \cite[Theorem~IX.3.1]{Mi} to see that \eqref{eq:FITgoupschemes} is an isomorphism.
\end{proof}

\begin{Remark}
One special case of this construction occurs when $\chi \in \g^*$ and $G$ acts on $(\g^*)^{(1)}$ via $$G \to G^{(1)} \to \End (\g^*)^{(1)},$$ the Frobenius morphism composed with the coadjoint representation of $G^{(1)}$. Then $G^{F_{\g^*}(\chi)}$ is the Frobenius neighbourhood of $G^\chi$ in $G$. This group scheme played a key role in the classification of Hamiltonian quantizations of $\GL_N$-orbits recently described by the second and third authors in \cite{ATW}, and will play a key role in the present work. We often write it simply as $G^{F(\chi)}$ for eacse of notation.

When $T \subseteq G$ is a maximal torus and $\chi \in \Lie(T)$ is a regular element, the group $G^{F(\chi)}$ is nothing but $G_1T$, which has appeared in various works, both classical and modern, in representation theory.
\end{Remark}

\section{Harish-Chandra bimodules}\label{s:HCbimods}

In this section we introduce the main objects of study throughout the paper: Harish-Chandra bimodules. We then derive some immediate structural features of these bimodules, including that $p$-support of a simple Harish-Chandra bimodules is a semisimple coadjoint orbit.

\subsection{Equivariant modules}
\label{ss:equivariantmodules}

Let $G$ be a group scheme which acts on an associative algebra $A$ by algebra automorphisms. A {\it weakly $G$-equivariant} (or just {\it $G$-equivariant}) $A$-module is an $A$-module equipped with the structure of a $G$-module such that the module map $A\otimes M\to M$ is $G$-equivariant. Equivalently, $M$ is an $A$-module in the category of $G$-modules. We denote the category of $G$-equivariant $A$-modules by $A\lmod^G$, and denote the category of $G$-equivariant $A$-modules which are finitely generated over $A$ by $A\lmod^G_{\fg}$. Note that such $M$ can be generated by a finite-dimensional $G$-submodule, since $G$ acts locally finitely.

If $\g = \Lie(G)$ and $A$ is equipped with a Lie algebra homomorphism $\g\to A$, we say that a $G$-equivariant $A$-module $M$ is {\it strongly $G$-equivariant} if the differential of the $G$-module structure coincides with $\g \to A \to \End(M)$. The category of strongly $G$-equivariant $A$-modules is denoted $(A,G)\lmod$, and we write $(A,G)\lmod_{\fg}$ for the subcategory of modules which are finitely-generated over $A$. If $A=U(\a)$, for a Lie algebra $\a$ containing $\g$ as a $G$-submodule, then objects in $(A,G)\lmod$ are usually called modules for the Harish-Chandra pair $(\a,G)$, cf. \cite[\textsection 5.2]{Et}.

If $A=R$ is a commutative algebra, then the categories $R\lmod$ and $R\lmod_{\fg}$ are $\k$-linear monoidal categories with tensor product $M\otimes_R N$. If $M$ and $N$ are $G$-equivariant $R$-modules then $M\otimes_R N$ can be given the structure of a $G$-equivariant $R$-module via the diagonal inclusion
\begin{eqnarray}
\label{eq:diagonalembedding}
G\hookrightarrow G\times G \text{ given by } g\mapsto (g,g).    
\end{eqnarray}
Thus, $R\lmod^G$ and $R\lmod_{\fg}^G$ are also monoidal categories.

Let $\a$ be a Lie algebra on which $G$ acts via automorphisms such that $\g$ is a $G$-submodule of $\a$, and let $\xi:\a\to A$ be a Lie algebra homomorphism. If $A=U(\a)$ or $U_0(\a)$ we may further equip $A\lmod$ with the structure of a $\k$-linear monoidal category by defining for $M,N\in A\lmod$ the vector space tensor product $M\otimes N$ to have $A$-module structure given by $$x\cdot (m\otimes n)=xm\otimes n + m\otimes xn \text{ for all } x\in\a.$$ This induces a monoidal structure on $A\lmod^G$ and $(A,G)\lmod$. In this case one may consider an algebra object $B$ in the category $(A, G)\lmod$ and define a notion of $(A, G)$-equivariant $B$-modules in the obvious way. This category will be denoted $B\lmod^{(A,G)}$.

\begin{Lemma}
\label{L:monoidalequivalence}
    Let $H$ be a subgroup of $G$  and let $K$ as defined in \eqref{e:defineK}. There is a equivalence of monoidal categories between $K\lmod$ and $(U_0(\g),H)\lmod$.
\end{Lemma}

\begin{proof}
    Observe first that Lemma~\ref{L:semidir} implies that there is a monoidal equivalence between $K\lmod$ and the category $\cC$ of $G_1\rtimes H$-modules on which $H_1$ acts trivially. We then define a functor $\cF:\cC\to (U_0(\g),H)\lmod $ where $\cF(M)=M$ as a $H$-module, and $\cF(M)$ is given the structure of a $U_0(\g)$-module via \eqref{eq:Frobeniuskernelandrestrictedenvelopingalgebra}. To check that $\cF(M)\in(U_0(\g),H)\lmod$, note that the $H$-equivariance of $U_0(\g)\otimes\cF(M)\to\cF(M)$ follows from $M$ being a module over $G_1\rtimes H$ and the strong equivariance follows from the triviality of the $H_1$-action.

    To show that $\cF$ is an equivalence of categories, observe that \eqref{eq:Frobeniuskernelandrestrictedenvelopingalgebra} easily implies it is fully faithful. For essential surjectivity, note that $M\in (H,U_0(\g))\lmod$ naturally has an $H$-module structure and can be equipped with a $G_1$-module structure using \eqref{eq:Frobeniuskernelandrestrictedenvelopingalgebra}; what remains is to check that these module structures induce a $G_1\rtimes H$-modules on which $H_1$ acts trivially. One can check that the $G_1\rtimes H$-module structure follows from $M$ being an $H$-equivariant $U_0(\g)$-module, whilst the $H_1$-trivially follows from the strongly $H$-equivariance.

    Finally, the fact that $\cF$ is an equivalence of monoidal categories follows from the fact that \eqref{eq:Frobeniuskernelandrestrictedenvelopingalgebra} is an equivalence of monoidal categories.
\end{proof}

\subsection{Harish-Chandra bimodules}
\label{ss:HCbimods}

Let $G$ be a group scheme and $\g = \Lie G$. The natural Hopf structure admits a comultiplication $\Delta$ and antipode $\omega$ determined, for $x\in \g$, by
\begin{eqnarray*}
    \Delta(x) & = & x\otimes 1 + 1\otimes x\\
    \omega(x) & = & -x
\end{eqnarray*}
 
Pulling back through $\Id\otimes \omega$ allows us to identify (left) $U(\g) \otimes U(\g)$-modules with $U(\g) \otimes U(\g)^\op$-modules, which are the same as $U(\g)$-$U(\g)$-bimodules. Throughout the paper, we will view bimodules in this way. We refer to the $\g$-action corresponding to $U(\g) \otimes 1$ and $1\otimes U(\g)$ as the {\it left} and {\it right} actions, respectively. 

Since $G$ acts on $U(\g)\otimes U(\g)$ via the diagonal embedding \eqref{eq:diagonalembedding}, we can consider the category $(U(\g) \otimes U(\g), G)\lmod$ of strongly $G$-equivariant $U(\g) \otimes U(\g)$-modules. The category $\HC^G U(\g)$ of Harish-Chandra bimodules is the full subcategory of $(U(\g) \otimes U(\g), G)\lmod$ whose objects are finitely generated on the left and (equivalently) on the right.

The strong equivariance condition means that, for $M \in \HC^G U(\g)$ with corresponding representation $\tau : U(\g) \otimes U(\g) \to \End(M)$, the differential of the $G$-action coincides with $\tau \circ \Delta|_{\g}$. 

More generally, $A$ is an algebra equipped with a Lie algebra homomorphism $\g \to A$ then the category $\HC^G A$ can be defined analogously.

As was observed in \cite[(3.5)]{BR} (for example) the forgetful functor 
    \begin{eqnarray}
        \label{L:HCequivalencewithequivariance}
    \HC^G U(\g) \isoto (U(\g)\otimes 1)\lmod^G_{\fg} = U(\g)\lmod^G_{\fg}
    \end{eqnarray}
   is an equivalence. The quasi-inverse functor can be understood as follows: the right action of $U(\g)$ on $M$ can be recovered from the left action and the differential of the $G$-action. When this action is recovered, strong $G$-equivariance is satisfied by construction (property (HC2) from Section~\ref{ss:thispaper}).

\subsection{The $p$-support of a Harish-Chandra bimodule} 
\label{ss:psuppHCbimods}

Now equip $\g$ with the natural $G$-equivariant $p$-mapping, and identify $Z_p(\g) = \k[(\g^*)^{(1)}]$. It is easy to check that the elements $x^p - x^{[p]} \in U(\g)$ are primitive
\begin{eqnarray}
    \label{eq:pcentreprimitives}
    \Delta(x^p - x^{[p]}) = (x^p - x^{[p]}) \otimes 1 + 1\otimes (x^p - x^{[p]}) \ \ \ \ \text{ for } x\in \g.
\end{eqnarray}

Let $M \in \HC^G U(\g)$ and let $\tau : U(\g) \otimes U(\g) \to \End M$ be the corresponding representation. Since $\tau \circ \Delta|_\g$ is the differential of the $G$-action, it is a restricted representation of $\g$, which means that $\tau \circ \Delta(x)^p = \tau \circ\Delta(x^{[p]})$ for all $x\in \g$. Equivalently,
\begin{eqnarray}
    \label{eq:leftequalsrightpchar}
\tau \circ \Delta|_{Z_p(\g)_+} = 0    
\end{eqnarray}
where $Z_p(\g)_+$ denotes the augmentation ideal of the $p$-centre, i.e. the maximal ideal of $0 \in (\g^*)^{(1)} = \Spec Z_p(\g)$. As a consequence, the right hand side of \eqref{eq:pcentreprimitives} acts trivially on $M$. This implies that $\Ann_{Z_p(\g) \otimes 1} M$ and $\Ann_{1\otimes Z_p(\g)} M$, viewed as ideals of $Z_p(\g)$, correspond bijectively under the antipode $\omega$.

The {\it $p$-support of $M$}, denoted $\Supp_p(M)$ is the subscheme of $(\g^*)^{(1)}$ with defining ideal $\Ann_{Z_p(\g) \otimes 1} M$. We could equally define support from the right action, but the previous remarks show that the schemes are isomorphic. We note that $\Supp_p(M)$ is $G$-stable.

 \begin{Lemma}
 \label{L:suportsemisimple}
    If $M \in \HC^G U(\g)$ is simple, then $\Supp_p(M)$ is a reduced scheme, and coincides with a closed $G$-orbit. In particular, if $G$ is standard reductive then $\Supp_p(M)$ is a semisimple orbit.
 \end{Lemma}
    \begin{proof}
        Write $I_M := \Ann_{Z_p(\g)\otimes 1} M \unlhd Z_p(\g)$. Since $M$ is simple, $I_M$ is radical. Hence the $p$-support is reduced and $G$-stable.
        
        Now suppose that $X \subsetneq \Supp_p(M)$ is a proper $G$-stable subscheme. Then $I_X M \subseteq M$ is a  submodule in $\HC^G U(\g)$. Using simplicity and the fact that $I_M \subsetneq I_X$ we have $I_XM = M$.

        Note that $U(\g)$, and hence $M$, is finitely generated over $Z_p(\g)$. Now $I_X M = M$ implies by Nakayama's lemma that there exists $r\in I_M$ with $r-1\in I_X$. This contradicts $I_M\subseteq I_X$ unless $X=\emptyset$, and so we deduce that $\Supp_p(M)$ is a closed $G$-orbit.
        
        The final claim follows from Lemma~\ref{L:Zariskiclosedsemisimple}.
    \end{proof}

For $X \subseteq (\g^*)^{(1)}$, a closed subscheme defined by $I_X \subseteq Z_p(\g)$, we write $\HC_X^G U(\g)$ for the category of Harish-Chandra bimodules with $\Supp_p(M) \subseteq X$. If $X$ is a closed subscheme in $\g^*$, we often abuse notation by writing $\HC_X^G U(\g)$ in place of $\HC_{X^{(1)}}^G U(\g)$ and $U_X(\g)$ in place of $U_{X^{(1)}}(\g)$, since this will not cause confusion. The equivalence \eqref{L:HCequivalencewithequivariance} then specialises to an equivalence of categories
\begin{eqnarray}
    \label{eq:HCequivalencewithequivariance+support}
    \HC_X^G U(\g) \isoto U_X(\g)\lmod^G_{\fg}.
\end{eqnarray} We particularly use this when $X$ is the closure of a $G$-orbit $\O$ in $\g^*$.

Recall the map $\pi$ from Subsection~\ref{ss:reductionsoverSteinbergfibres}. If $\O \subseteq \g^*$ is any orbit then there exists some $\lambda \in \t^*$ such that $\O \subseteq \pi^{-1}[\lambda]$. By Theorem~\ref{T:blockdecomp} we have an isomorphism
\begin{equation} \label{eq: phi eta orbit}
    \phi \colon U_{\overline\O}(\g) \to \bigoplus_{\psi \in \AS^{-1}[\lambda]} U_{\overline\O}^{(\psi)}(\g),
\end{equation}
where $U_{\overline\O}^{(\psi)}(\g)$ is the quotient of $U_{X}^{(\psi)}(\g)$ by $I_{\overline\O}U_{X}^{(\psi)}(\g)$. For $\psi_1, \psi_2 \in \AS^{-1}[\lambda]$ define $\HC_{\overline\O, (\psi_1, \psi_2)}^G U(\g) \subseteq \HC_{\overline\O}^G U(\g)$ to be the full category consisting of $M$ such that for all $m\in M$
\begin{itemize}
    \item[(i)] there exists $i\gg 0$ such that $(\ker \psi_1 \otimes 1)^i m = 0$;
    \item[(ii)] there exists $i\gg 0$ such that $(1\otimes \ker \psi_2)^im = 0$.
\end{itemize}
\begin{Corollary}
\label{C:centralcharactersindecomposable}
    We have an equivalence of categories
    \begin{eqnarray}
        \HC_{\overline\O}^G U(\g) \isoto \bigoplus_{\psi_1, \psi_2 \in \AS^{-1}(\lambda)} \HC_{\overline\O, (\psi_1, \psi_2)}^G U(\g).
    \end{eqnarray}
    Consequently, every indecomposable Harish-Chandra bimodule in $\HC_\cO^G U(\g)$ has a left and right generalised Harish-Chandra central character.
\end{Corollary}
\begin{proof}
    Since $\HC_{\overline\O}^G U(\g)$ is a full subcategory of $U_{\overline\O}^{(\psi_1)}(\g) \otimes U_{\overline\O}^{(\psi_2)}(\g)\lmod^G$, this is a standard consequence of \eqref{eq: phi eta orbit}, see \cite[II.5]{ARS} for example.
\end{proof}

\section{Descent for Harish-Chandra bimodules}
\label{S:descent}

The goal of this section is to derive an equivalence of categories $U_\O(\g)\lmod^G_{\fg}\to U_\chi(\g)\lmod^{G^{F(\chi)}}_{\fg}$ for $\O$ a semisimple $G$-orbit in $\g^*$ and $\chi\in\O$. We do this through the theory of descent. After stating some definitions in Subsection~\ref{ss:Gequivsheaves}, we recall the general theory of descent in Subsection~\ref{ss:descent} and then apply it to Harish-Chandra bimodules in Subsection~\ref{ss:applicationtoHCbimodules}. This latter subsection also considers orbits which are not semisimple, in which case we instead obtain an equivalence between $U_\chi(\g)\lmod^{G^{F(\chi)}}_{\fg}$ and some Serre quotient of $U_\cO(\g)\lmod^G_{\fg}$.

\subsection{$G$-equivariant coherent sheaves}\label{ss:Gequivsheaves}
Let $X$ be a scheme. We recall the notion of equivariant coherent sheaves on $X$.

The category $\Coh(X)$ of coherent $\Oc_X$-modules is a $\k$-linear monoidal category, with tensor product $\cF\otimes_{\Oc_X}\cG$ defined to the sheafification of the presheaf $U\mapsto \cF(U)\otimes_{\Oc_X(U)}\cG(U)$. When $X$ is affine, $\Coh(X)$ is equivalent to the category $\k[X]\lmod_{\fg}$ of finitely-generated $\k[X]$-modules. This is in fact an equivalence of monoidal categories, where the monoidal structure on $\k[X]\lmod_{\fg}$ is given by $(M,N)\mapsto M\otimes_{\k[X]}N$ (see \cite[Proposition 5.2]{Ha}).

Fix now a group scheme $G$. We say that $X$ is a $G$-scheme if it is equipped with a morphism $\sigma:G\times X\to X$ satisfying the usual conditions for a group action; for example, $\sigma\circ(m\times 1)=\sigma\circ(1\times \sigma)$ where $m:G\times G\to G$ is the group multiplication. The coordinate ring $\k[X]$ of a $G$-scheme $X$ is naturally equipped with an action of $G$ by algebra automorphisms. For the rest of the current section, $X$ is a $G$-scheme.

Write $p_2:G\times X\to X$ for the morphism $(g,x)\mapsto x$ and $p_{23}:G\times G\times X\to G\times X$ for the morphism $(g,h,x)\mapsto (h,x)$. A $G$-equivariant sheaf on a $G$-scheme $X$ is then a pair $(\cF,\theta)$, where $\cF$ is an $\Oc_X$-module and $\theta:\sigma^*\cF\isoto p_2^*\cF$ is an isomorphism of $\Oc_{G\times X}$-modules such that 
\begin{equation}\label{eq:Gequivrel}
p_{23}^*\theta\circ (1\times\sigma)^*\theta =(m\times 1)^*\theta.
\end{equation}
A morphism of $G$-equivariant sheaves $(\cF,\theta_\cF)\to (\cG,\theta_{\cG})$ is a morphism of $\Oc_X$-modules $\cF\to\cG$ such that the following diagram commutes
\begin{eqnarray}
\begin{array}{c}\xymatrix{
     \sigma^*\cF \ar@{->}[rr]  \ar@{->}[d]^{\theta_\cF} & & \sigma^*\cG \ar@{->}[d]_{\theta_\cG}
     \\
     p_2^*\cF \ar@{->}[rr] & & p_2^*\cG
    }
\end{array}
\end{eqnarray}

We denote by $\Coh^G(X)$ the category of $G$-equivariant sheaves of coherent $\Oc_X$-modules. This is a monoidal category, where the tensor product $\cF\otimes_{\Oc_X}\cG$ is given the structure of a $G$-equivariant sheaf via the isomorphism $$\theta_{\cF\otimes\cG}:\sigma^*(\cF\otimes_{\Oc_X}\cG)=\sigma^*\cF\otimes_{\Oc_{G\times X}}\sigma^*\cG\xrightarrow{\theta_\cF\otimes \theta_\cG} p_2^*\cF\otimes_{\Oc_{G\times X}}p_2^*\cG=p_2^*(\cF\otimes_{\Oc_X}\cG).$$ 

When $X$ is affine, the monoidal equivalence of categories $\Coh(X)\to\k[X]\lmod_{\fg}$ can be upgraded to a monoidal equivalence of categories
\begin{equation}\label{eq:mod=qcoh}\Coh^G(X)\isoto \k[X]\lmod^G_{\fg}.\end{equation}

In particular, Lemma~\ref{L:Falgisalg} shows that an algebra in $\k[X]\lmod^G_{\fg}$ gives rise to an algebra in $\Coh^G(X)$. Note that an algebra in the category $\k[X]\lmod^G$ is precisely a $G$-equivariant $\k[X]$-algebra, and that an algebra in the category $\Coh^G(X)$ is a $G$-equivariant sheaf of $\Oc_X$-algebras which are coherent as $\Oc_X$-modules and for which the structure map $\theta:\sigma^*\cF\isoto p_2^*\cF$ is an isomorphism of $\Oc_{G\times X}$-algebras. 

Let us adopt the notational convention that a $G$-equivariant $\k[X]$-algebra $A$ gives rise to a $G$-equivariant sheaf $\cA$ of coherent $\Oc_X$-algebras. Lemma~\ref{L:Fmodismod} then yields an equivalence of categories
\begin{equation}\label{eq:localAmodv1}
    \cA\lmod_{\Coh^G(X)} \isoto A\lmod_{\k[X]\lmod^G}.
\end{equation} The category $A\lmod_{\k[X]\lmod^G_{\fg}}$ coincides with the category $A\lmod^G$ of $G$-equivariant $A$-modules. Furthermore, the category $\cA\lmod_{\Coh^G(X)}$ is precisely the category of $G$-equivariant sheaves $\cF$ of $\cA$-modules which are coherent as $\Oc_X$-modules and for which the structure map $\theta_\cF:\sigma^*\cF\isoto p_2^*\cF$ is an isomorphism of $\Oc_{G\times X}$-algebras such that the following diagram commutes
\begin{eqnarray}
\begin{array}{c}\xymatrix{
     \sigma^*\cA\otimes_{\Oc_{G\times X}}\sigma^*\cF \ar@{->}[rr]  \ar@{->}[d]^{\theta_\cA\otimes\theta_\cF} & & \sigma^*(\cA\otimes_{\Oc_X}\cF) \ar@{->}[rr]  & & \sigma^*\cF \ar@{->}[d]_{\theta_\cF}
     \\
     p_2^*\cA\otimes_{\Oc_{G\times X}}p_2^*\cF \ar@{->}[rr] & & p_2^*(\cA\otimes_{\Oc_X}\cF) \ar@{->}[rr] & &  p_2^*\cF
    }
\end{array}
\end{eqnarray}
We denote this category $\cA\lmod^G_{\fg}$. In particular, we may rewrite \eqref{eq:localAmodv1} as 
\begin{equation}\label{eq:locofAmod}
    A\lmod^G_{\fg}\isoto \A\lmod^G_{\fg}
\end{equation}

\subsection{Descent}\label{ss:descent}
In this subsection, we describe a manifestation of faithfully flat descent for coherent sheaves.

Let $G$ be a connected algebraic group and let $H$ be a subgroup scheme in $G$. We may view $G$ as a $G\times H$-scheme, via left and right multiplication, and thus consider $\Coh^{G\times H}(G)$. The quotient $X:=G/H$ is naturally a $G$-scheme \cite[I.5.6(8)]{JanRAGS} and the natural morphism $\pi:G\to G/H$ is $G$-equivariant.

Now we apply faithfully flat descent \cite[Expos\'{e} VIII, Corollaire 1.3]{SGA} to our context. We also direct the reader to \cite[Lemma 2.3]{Br}, \cite[Theorem 4.2.14]{HL} or \cite[Theorem~2.7]{CPS} for quite different approaches to the theory. Descent yields an equivalence of categories
\begin{eqnarray}
\label{eq:equiv1}
\Coh^G(X)\isoto \Coh^{G\times H}(G), \qquad \cF\mapsto \pi^*\cF    
\end{eqnarray}
with quasi-inverse $\cG\mapsto \pi_*^H\cG$. Note that $\pi^*\cF$ can be naturally equipped with an $H$-equivariant structure via the identity $\sigma^*\pi^*\cF=(\pi\circ\sigma)^*\cF=(\pi\circ p_2)^*\cF=p_2^*\pi^*\cF$. The functor $\cF\mapsto \pi^*\cF$ is monoidal (see \cite[Lemma 1.17.16.4]{St}), and thus the quasi-inverse $\cG\mapsto \pi_*^H\cG$ may be given the structure of a monoidal functor as in \cite[Remark 2.4.10]{EGNO}.

Writing $\eta$ for the (unique) map $G\to G/G=\{\ast\}$ and using faithfully flat descent again, we similarly obtain an equivalence of categories
\begin{eqnarray}
\label{eq:equiv2}
    \Coh^{G\times H}(G)\isoto \Coh^{H}(\{\ast\})=H\lmod_{\fd}, \qquad \cG\mapsto \cG(G)^G
\end{eqnarray}
with quasi-inverse $\cF\mapsto \eta^*\cF$. In this context, this quasi-inverse is the functor sending the $H$-module $M$ to $\Oc_G\otimes_{\underline{\k}}\underline{M}$, where $\underline{\k}$ and $\underline{M}$ denote constant sheaves. As before, $\cF\mapsto \eta^*\cF$ is monoidal and thus we may equip $\cG\mapsto \cG(G)^G$ with the structure of a monoidal functor.

Combining \eqref{eq:equiv1} and \eqref{eq:equiv2} yields the following proposition.
\begin{Proposition}\label{p:Qcoh=Hmod}
There is an equivalence of monoidal categories \begin{equation}\label{eq:qc=Hmod}
\Coh^G(X)\isoto H\lmod_{\fd}, \qquad \cF\mapsto (\k[G]\otimes_{\k[X]}\cF(X))^G
\end{equation} with (monoidal) quasi-inverse $M\mapsto \pi_*^H(\Oc_G\otimes_{\underline{\k}}\underline{M})$. 
\end{Proposition}

In our work we will need to describe the functor of categories \eqref{eq:qc=Hmod} in an alternative manner. Denote by $\iota$ the $H$-equivariant morphism $\{\ast\}\to X$ which sends $\ast$ to $x_0:=\pi(1_G) \in X$. For $\cF\in\Coh(X)$, the pullback $\iota^*\cF\in \Coh^H(\{\ast\})=H\lmod_{\fd}$ is precisely
\begin{eqnarray}
    \iota^*\cF(\{\ast\})=\k\otimes_{\k[X]/\m_{x_0}} \cF(X)_{\m_{x_0}}=\cF(X)/\m_{x_0}\cF(X).
\end{eqnarray} 
As in \cite[Lemma 2.4]{Br}, $\iota^*$ is a quasi-inverse to the functor $\pi_*^H\eta_*:H\lmod_{\fd}\to\Coh^G(X)$. Indeed, writing $j:\{\ast\}\to G$ for the morphism sending $\ast$ to $1_G$ and observing that $\iota=\pi\circ j$, we easily obtain that $$\iota^*\pi_*^H\eta^*=j^*\pi^*\pi_*^H\eta^*=j^*\eta^*=(\eta\circ j)^*=\Id.$$ Therefore, for any $\cF\in\Coh^G(X)$, we have $$\iota^*\cF\simeq(\eta_*^G \pi^*)(\pi_*^H\eta^*)\iota^*\cF\simeq \eta_*^G\pi^*\cF.$$ 

In particular, we may reinterpret Proposition~\ref{p:Qcoh=Hmod} as follows.
\begin{Corollary}\label{c:kXmod=Hmod}
There is an equivalence of monoidal categories \begin{equation}\label{eq:qc=Hmodv2}
\Coh^G(X)\isoto H\lmod_{\fd}, \qquad \cF\mapsto \cF(X)/\m_{x_0}\cF(X)
\end{equation} with (monoidal) quasi-inverse $M\mapsto \pi_*^H(\Oc_G\otimes_{\underline{\k}}\underline{M})$. 
\end{Corollary}
This result can also be found as \cite[Theorem 2.7, Corollary 2.10]{CPS}.

Combining this corollary with \eqref{eq:mod=qcoh} then yields the following result when $X$ is affine.

\begin{Corollary}\label{c:kXmod=Hmodv2}
    Suppose that $X=G/H$ is affine. Then there exists an equivalence of monoidal categories
    \begin{equation}\label{eq:kXmodG=Hmodv2}
    \k[X]\lmod^G_{\fg}\isoto H\lmod_{\fd},\qquad V\mapsto V/\m_0 V,  
    \end{equation} with quasi-inverse $M\mapsto (\k[G]\otimes M)^H$. $\hfill\qed$
\end{Corollary}

\subsection{Application to Harish-Chandra bimodules}
\label{ss:applicationtoHCbimodules}

Let $G$ be a connected standard reductive group as in Subsection~\ref{ss:reductivegroupsandLiealgebras}, and let $\chi\in\g^*$. Write $F=F_{\g^*}$ for the Frobenius morphism on $\g^*$, and let $G^{F(\chi)}$ be as defined in Subsection~\ref{ss:stabilisersoftwistedcharacters}. By \cite[I.5.6(2)]{JanRAGS} we may identify $G/G^{F(\chi)}$ with $\O^{(1)}$, a subscheme of $(\g^*)^{(1)}$ (see also \cite[Proposition I.9.5]{JanRAGS} and \cite[\textsection 2]{JaNO}). 

Let $I_{\o \O}\lhd Z_p(\g)$ be the defining ideal of $\o \O^{(1)}$ in $Z_p(\g)\cong \k[(\g^*)^{(1)}]$. Then $Z_p(\g)/I_{\o \O}=\k[\o \O^{(1)}]$ and therefore $U_{\o \O}(\g)$ is a finite $\k[\o \O^{(1)}]$-algebra. In fact, since $\o \O$ is $G$-stable we obtain that $U_{\o \O}(\g)$ is an algebra in the category $\k[\o \O^{(1)}]\lmod^G_{\fg}$. Under the monoidal equivalence \eqref{eq:mod=qcoh}, the algebra $U_{\o \O}(\g)$ corresponds to an algebra in the category $\Coh^G(\o \O^{(1)})$. We denote this resulting sheaf of algebras by $\U_{\o \O}(\g).$ Then \eqref{eq:locofAmod} yields an equivalence of categories 
\begin{equation}\label{eq:locofUmod}
    U_{\o \O}(\g)\lmod^G_{\fg}\isoto \U_{\o \O}(\g)\lmod^G_{\fg}
\end{equation}
Writing $\varepsilon:\O^{(1)}\hookrightarrow \cO^{(1)}$ for the natural $G$-equivariant embedding, we further make the definition $$\U_\O(\g):=\varepsilon^*\U_{\o \O}(\g).$$ This is an algebra in $\Coh^G(\O^{(1)})$.
\begin{Proposition}\label{P:MainEquiv}
    There exists an equivalence of categories 
    \begin{equation}\label{eq:descUg}
        \U_\O(\g)\lmod^G_{\fg}\isoto U_\chi(\g)\lmod^{G^{F(\chi)}}_{\fg}.
    \end{equation}
\end{Proposition}
\begin{proof}
    Applying the monoidal equivalence \eqref{eq:qc=Hmodv2} to $\U_\O(\g)\in\Coh^G(\O^{(1)})$ yields $U_\chi(\g)\in G^{F(\chi)}\lmod_{\fg}$. Therefore, the desired equivalence is a consequence of Lemma~\ref{L:Fmodismod}.
\end{proof}

\begin{Corollary}\label{c:locofUXmod}
    Suppose that $\O$ is a closed $G$-orbit in $\g^*$. Then there exists an equivalence of categories 
    \begin{equation}\label{eq:mainequiv}
    U_\O(\g)\lmod^G_{\fg}\isoto U_\chi(\g)\lmod^{G^{F(\chi)}}_{\fg}.
    \end{equation}
\end{Corollary}
\begin{proof}
    Since $\O$ is closed, the embedding $\varepsilon$ is the identity map. The result follows by combining \eqref{eq:locofUmod} and Proposition~\ref{P:MainEquiv}.
\end{proof}

If $\O$ is not closed in $\g^*$ then $\varepsilon^*:\Coh^G(\cO^{(1)})\to \Coh^G(\O^{(1)})$ only induces a functor $\o\varepsilon^*:\U_{\cO}(\g)\lmod^G_{\fg}\to \U_\O(\g)\lmod^G_{\fg}$. Denote by $\cK$ the kernel of this functor, i.e. the Serre subcategory of $\U_{\o \O}(\g)\lmod^G_{\fg}$ generated by those objects $\cF$ such that $\o\varepsilon^*\cF\cong 0$. We then have the following result.

\begin{Proposition}
\label{P:generalequivalencebydescent}
    The exact functor $\varepsilon^*$ induces an equivalence of categories $$\frac{\U_{\o \O}(\g)\lmod^G_{\fg}}{\cK}\isoto \U_\O(\g)\lmod^G_{\fg}.$$
\end{Proposition}
\begin{proof}
    By \cite[Proposition III.2.5]{Ga}, it suffices to show that the exact functor $\o\varepsilon^*:\U_{\cO}(\g)\lmod^G_{\fg}\to \U_\O(\g)\lmod^G_{\fg}$ admits a fully faithful right adjoint. Note that since $\k[\O^{(1)}]$ is a finitely generated $\k[\cO^{(1)}]$-module (see \cite[Proposition 8.3]{JaNO}), the push-forward functor $\varepsilon_*:\QCoh^G(\O^{(1)})\to \QCoh^G(\cO^{(1)})$ preserves coherence. In particular, $\varepsilon_*:\Coh^G(\O^{(1)})\to\Coh^G(\cO^{(1)})$ is a fully faithful right adjoint to $\varepsilon^*$, and while it is not monoidal it does come equipped with a natural transformation $$\hat{J}:\varepsilon_*(-)\otimes_{\Oc_{\cO^{(1)}}}\varepsilon_*(-)\to \varepsilon_*(-\otimes_{\Oc_{\O^{(1)}}}-)$$ which satisfies the conditions of a monoidal structure action. Furthermore, the unit $\eta:\Id\to \varepsilon_*\varepsilon^*$ and counit $\gamma:\varepsilon^*\varepsilon_*\to \Id$ of adjunction satisfy the commutative diagram required for a natural transformation of monoidal functors. 
    
    In particular, given $(\cF,\rho)\in\U_\O(\g)\lmod^G_{\fg}$ we may form a $\U_{\cO}(\g)$-module $\varepsilon_*\cF$ with module structure given by $$\U_\cO(\g)\otimes_{\Oc_{\cO^{(1)}}}\varepsilon_*\cF\xrightarrow{\eta_{\U_{\tiny\cO}(\g)}\otimes 1} \varepsilon_* \varepsilon^* \U_\cO(\g)\otimes_{\Oc_{\cO^{(1)}}}\varepsilon_*\cF \xrightarrow{\hat{J}_{\varepsilon^*\U_{\tiny\cO}(\g),\cF}} \varepsilon_*(\U_\O(\g)\otimes_{\Oc_{\O^{(1)}}}\cF)\xrightarrow{\varepsilon_*(\rho)} \varepsilon_*\cF.$$ This defines a functor $$\o\varepsilon_*:\U_\O(\g)\lmod^G_{\fg}\to \U_\cO(\g)\lmod^G_{\fg}.$$ One can easily check that this functor is right adjoint to $\o\varepsilon^*$ and is fully faithful (or, equivalently, that the counit of adjunction $\o\varepsilon^*\o\varepsilon_*\to \Id$ is an isomorphism). This proves the result.
\end{proof}

We conclude this subsection with two alternative descriptions of the kernel $\cK$. To state it, let us define the boundary of $\cO$ by $\delta\cO:=\cO\setminus \O$. This is a $G$-stable, closed subvariety in $\g^*$. Write $\omega:\delta \cO^{(1)} \hookrightarrow \cO^{(1)}$ for the natural inclusion; this defines a functor $\omega_*:\Coh^G(\delta\cO^{(1)})\to \Coh^G(\cO^{(1)})$ and a functor $\o\omega_*:\U_{\delta\cO}(\g)\lmod^G_{\fg}\to \U_{\cO}(\g)\lmod^G_{\fg}$.

\begin{Lemma}
\label{L:tiptop}
    The following three Serre subcategories of $\U_\cO(\g)\lmod^G_{\fg}$ coincide:
    \begin{enumerate}
        \item $\cK$, the kernel of $\o\varepsilon^*$.
        \item The full subcategory consisting those of $\cF$ supported in $\delta\cO^{(1)}.$
        \item The Serre subcategory generated by objects in the essential image of $\o\omega_*$. 
    \end{enumerate}
\end{Lemma} 

\begin{proof}
    Let us denote the Serre subcategories listed in the statement of the lemma by $\H_1$, $\H_2$ and $\H_3$ respectively. We prove that each object in $\H_1$ lies in $\H_2$, that each object of $\H_2$ lies in $\H_3$, and that each object in $\H_3$ lies in $\H_1$.

    {\it Step 1:} Each object in  $\H_1$ lies in $\H_2$.

    Let $\cF\in\H_1$, so $\o\varepsilon\cF\iso0$; in particular, this means $\varepsilon^*\cF=0$. For $x\in \O^{(1)}$, define $i_x:\{*\}\to \O^{(1)}$ to be the morphism sending $*$ to $x$. Then $$0=i_x^*\varepsilon^*\cF(*)=(\varepsilon i_x)^*\cF(*)=\frac{\cF(\cO^{(1)})_{\m_x}}{\m_x\cF(\cO^{(1)})_{\m_x}},$$ where $\m_x$ is the maximal ideal of $\k[\cO^{(1)}]$ defining $x$. By \cite[29.5.3]{St}, this implies $\cF_x=0$ for all $x\in\O^{(1)}$, and therefore that $\cF\in \H_2$.

    {\it Step 2:} Each object in  $\H_2$ lies in $\H_3$.
    
    Let $\cF\in\H_2$. By \cite[29.5.4]{St}, there exists $I\in \k[\cO^{(1)}]$ with $\k[\cO^{(1)}]/\sqrt{I}=\k[\delta\cO^{(1)}]$ such that $\cF(\cO^{(1)})$ is a $G$-equivariant $U_\cO(\g)/I U_\cO(\g)$-module. There exists $k$ such that $\sqrt{I}^k\subseteq I$ and therefore $\cF(\cO^{(1)})$ is a $G$-equivariant $U_\cO(\g)/\sqrt{I}^k U_\cO(\g)$-module. Consider the filtration $$\cF(\cO^{(1)})\supseteq \sqrt{I}\cF(\cO^{(1)})\supseteq \sqrt{I}^2\cF(\cO^{(1)})\supseteq \cdots\supseteq \sqrt{I}^k\cF(\cO^{(1)})=0.$$  Each subquotient is a $G$-equivariant $U_\cO(\g)/\sqrt{I} U_\cO(\g)$-module, i.e. a $G$-equivariant $U_{\delta\cO}(\g)$-module. The claim follows.
    
    {\it Step 3:} Each object in  $\H_3$ lies in $\H_1$.

    Since $\cK$ is a Serre subcategory of $\U_\cO(\g)\lmod^G_{\fg}$, it suffices to show that $\o\omega_*\cF\in\H_1$ for all $\cF\in\U_{\delta\cO}(\g)\lmod^G_{\fg}$. For this, we need only show that $\varepsilon^*\omega_*:\Coh^G(\delta\cO^{(1)})\to \Coh^G(\O^{(1)})$ is zero, which is an easy calculation. 
\end{proof}

\section{An equivariant Friedlander-Parshall equivalence}
\label{S:centralreductions}

In Section~\ref{ss:applicationtoHCbimodules} we effectively showed that the study of simple Harish-Chandra bimodules can be reduced to the study of simple objects in $U_\chi(\g)\lmod^{G^{F(\chi)}}$ for $\chi\in\g^*$ semisimple. In this section, we reduce the discussion further to the setting of restricted Harish-Chandra bimodules over a Levi subalgebra. The reduction largely follows the approach of Friedlander--Parshall \cite[Theorem~3.2]{FP},  upgrading their celebrated category equivalence to the setting of equivariant bimodules, see also \cite[\textsection 7.4]{JaLA} for a nice exposition.

Although we only need these results when $\chi$ is semisimple, we in fact prove them for arbitrary $\chi$ since the arguments are not substantially more complicated.

\subsection{Equivariant reduction to nilpotent $p$-characters}
\label{ss:equivariantreductiontonilpotent}

Let $\chi \in \g^*$, with Levi decomposition $\chi = \chi_\rs + \chi_\rn$, and let $H \subseteq G^\chi$ be a closed subgroup. As noted in Lemma~\ref{L:Zariskiclosedsemisimple}, $G^{\chi_\rs}$ is a Levi subgroup of $G$ and so we can choose a parabolic subgroup $P \subseteq G$ containing $G^{\chi_\rs}$ as a Levi factor. We denote the unipotent radical of $P$ by $U$ and write $\p = \g^{\chi_\rs} \oplus \u$ for the corresponding Lie algebras -- we note that $\g^{\chi_\rs} = \Lie G^{\chi_\rs}$ by \cite[2.2(4)]{JaNO}.

Recall axiom (H3) from Section~\ref{ss:reductivegroupsandLiealgebras}, which states that the adjoint $G$-module is self-dual. Hence there exists a $G$-equivariant isomorphism $\g \to \g^*$ which induces a non-degenerate $G$-invariant form $\kappa : \g \times \g \to \k$. Suppose that $\chi = \kappa(x, \cdot)$ for $x\in \g$. Then $\kappa|_{\g^\chi \times\g^\chi}$ is non-degenerate whilst $\u$ lies in non-degenerate pairing with the nilradical of the parabolic opposite to $\p$ (see \cite[6.6]{JaLA}, for example). Since $x\in \g^x = \g^\chi$ it follows that $\chi(\mathfrak{u}) = 0$.

We can extend any $U_\chi(\g^{\chi_s})$-module to a $U_\chi(\mathfrak{p})$-module, letting $\u$ act trivially. For $M \in U_\chi(\g)\lmod$ we write $M^\u$ for the $\u$-invariants.

We claim that the following functors are well defined
\begin{equation} \label{eq: FP equivalence 1 H}
    \begin{array}{rclcrcl}
    U_\chi(\g)\lmod_{\fg}^H & \to & U_\chi(\g^{\chi_s})\lmod_{\fg}^H & & M & \mapsto &M^{\u}\vspace{8pt} \\
    U_\chi(\g^{\chi_s})\lmod_{\fg}^H & \to & U_\chi(\g)\lmod_{\fg}^H & &  N & \mapsto & \Ind(N) := U_\chi(\g) \otimes_{U_\chi(\mathfrak{p})} N.
    \end{array}
\end{equation}

If $H$ is the trivial group then these functors are just the ones considered by Friedlander--Parshall \cite{FP}. Since $H \subseteq G^\chi \subseteq P$ and $P$ acts on $\u$ it follows that $M^\u$ is $H$-stable, so that the first functor in \eqref{eq: FP equivalence 1 H} is well-defined.

The action of $H$ on $U_\chi(\g) \otimes_{U_\chi(\mathfrak{p})} N$ is diagonal, i.e. $h\cdot (u \otimes n) = (h\cdot u) \otimes (h\cdot n)$, and this is well-defined because $U_\chi(\p)$ is $H$-stable. Now for $h\in H$, $x, y \in U_\chi(\g)$ and $n\in N$, we have
\[
    h \cdot (x \cdot (y \otimes n)) = h \cdot (xy \otimes n) = (h \cdot x)(h \cdot y) \otimes (h \cdot n) = (h \cdot x) \cdot ((h \cdot y) \otimes (h \cdot n))
\]
as required, so the $U_\chi(\g)$-action is $H$-equivariant and the second functor \eqref{eq: FP equivalence 1 H} is well-defined.

\begin{Lemma} \label{thm: FP equivalence}
    If $H \subseteq G^\chi$ is a closed subgroup scheme then the functors \eqref{eq: FP equivalence 1 H} give an equivalence of categories
    \begin{equation}
        U_\chi(\g)\lmod^H_{\fg} \isoto U_\chi(\g^{\chi_s})\lmod^H_{\fg}.
    \end{equation}
\end{Lemma}

\begin{proof}
    Note that $\Ind$ is left adjoint to $(\cdot)^\u$.
    To see this, take $M \in U_\chi(\g)\lmod^H_{\fg}$ and $N \in U_\chi(\g^{\chi_s})\lmod^H_{\fg}$.
Then we have natural maps 
\begin{equation}
    \label{eq:refme}
    \Hom_{U_\chi(\g^{\chi_\rs})\lmod^H_{\fg}}(N, M^\u) \underset{\eta}{\overset{\theta}{\myrightleftarrows{\rule{1cm}{0cm}}}} \Hom_{U_\chi(\g)\lmod^H_{\fg}}(\Ind(N), M) 
\end{equation}
given by $\theta(f) (u \otimes n) = u \cdot f(n)$ and $\eta(g)(n) = g(1 \otimes n)$. Indeed, a short calculation shows that $\theta(f)$ is $U_\chi(\g)$-equivariant, that $\eta(g)$ is $U_\chi(\g^{\chi_\rs})$-equivariant, and that both are $H$-equivariant. Furthermore, it is straightforward to check that $\theta$ and $\eta$ are mutually inverse linear maps, hence isomorphisms.

    Since the two functors are adjoint, it is enough to show that the corresponding unit and counit maps are natural isomorphisms (see \cite[IV.4 Theorem 1]{MacLane}, for example).  In particular, we need to show that the maps
    \begin{equation*}
        U_\chi(\g) \otimes_{U_\chi(\p)} M^\u \to M \qquad u \otimes m \mapsto u \cdot m,
    \end{equation*}
    and
    \begin{equation*}
        N \to (U_\chi(\g) \otimes_{U_\chi(\p)} N)^\u, \qquad n \mapsto 1 \otimes n
    \end{equation*}
    are isomorphisms for all $M \in U_\chi(\g)\lmod^H_{\fg}$ and $N \in U_\chi(\g^{\chi_s})\lmod^H_{\fg}$. Since these maps are $H$-equivariant is suffices to know that the maps are isomorphisms of finitely generated $U_\chi(\g)$-modules and $U_\chi(\g^{\chi_\rs})$-modules, respectively. The latter follows from \cite[Theorem~3.2]{FP}.
\end{proof}

Now we upgrade the equivalence of Lemma~\ref{thm: FP equivalence} to the setting of Harish-Chandra bimodules. Keep fixed $\chi \in \g^*$, and observe that $\HC_\chi^{G^\chi} U(\g)$ is a full subcategory of $U_\chi (\g) \otimes U_{-\chi}(\g) \lmod^{G^\chi} \cong U_{(\chi, -\chi)} (\g \oplus \g) \lmod^{G^\chi}$, where $G^\chi$ is diagonally embedded into $G^\chi \times G^\chi$. Similarly $\HC^{G^\chi}_{\chi} U(\g^{\chi_s})$ is a full subcategory of $U_{(\chi, -\chi)} (\g^{\chi_s} \oplus \g^{\chi_s}) \lmod^{G^\chi}$.

Lemma~\ref{thm: FP equivalence} gives an equivalence of categories
\begin{eqnarray}
    \label{eq:consequentequivalence}
    U_\chi (\g) \otimes U_{-\chi}(\g) \lmod^{G^\chi} & {\myrightleftarrows{\rule{1cm}{0cm}}} & U_{\chi} (\g^{\chi_\rs}) \otimes U_{-\chi}(\g^{\chi_s}) \lmod^{G^\chi},
\end{eqnarray}
determined on objects of the categories by
    \begin{eqnarray*}
            M & \longmapsto & M^{\u \oplus \u}\\
    (U_\chi(\g) \otimes U_{-\chi}(\g))\otimes_{U_{\chi}(\p) \otimes U_{-\chi}(\p)} N & \longmapsfrom & N.
    \end{eqnarray*}

Given $M \in \HC_\chi^{G^\chi} U(\g)$, it is clear that $M^{\u \oplus \u}\in \HC_\chi^{G^\chi} U(\g^{\chi_\rs})$ since it is a sub-bimodule of $M$.

We claim that, for $N \in \HC_\chi^{G^\chi} U(\g^{\chi_s})$, the induced module $U_{(\chi, -\chi)}(\g \oplus \g) \otimes_{U_{(\chi, -\chi)} (\p \oplus \p)} N$ is also a Harish-Chandra bimodule. To check the claim, let $x\in \Lie(G^\chi)$, pick $u_1 \otimes u_2 \in U_\chi(\g) \otimes U_{-\chi}(\g)$ and $n \in N$. Writing $\Ad : G \to \End(N)$, the strong equivariance condition for $\Ind(N)$ requires
\begin{equation}
\label{eq:checkme}
    (x \otimes 1 + 1 \otimes x) \cdot ((u_1 \otimes u_2) \otimes_{U(\p)} n) = d_1 \Ad (x) ((u_1 \otimes u_2) \otimes_{U(\p)} n).
\end{equation}

Expanding the right-hand side we get
\begin{align*}
    d_1 \Ad(x) ((u_1 \otimes u_2) \otimes_{U(\p)} n) &= (d_1 \Ad(x) (u_1 \otimes u_2)) \otimes_{U(\p)} n + (u_1 \otimes u_2) \otimes_{U(\p)} (d_1 \Ad(x) n). \\
\intertext{The first term on the right hand side is}
    (d_1 \Ad(x) (u_1 \otimes u_2)) \otimes_{U(\p)} n &= ((x \otimes 1 + 1 \otimes x) \cdot (u_1 \otimes u_2) )\otimes_{U(\p)} n \\
    &= (xu_1 \otimes u_2 - u_1 x \otimes u_2 + u_1 \otimes x u_2 - u_1 \otimes u_2 x) \otimes_{U(\p)} n,
\end{align*}
whilst the second term is
\begin{align*}
    (u_1 \otimes u_2) \otimes_{U(\p)} d_1 \Ad(x) n &= (u_1 \otimes u_2) \otimes_{U(\p)} ((x \otimes 1 + 1 \otimes ) \cdot n) \\
    &= (u_1 \otimes u_2)(x \otimes 1 + 1 \otimes x) \otimes_{U(\p)} n \\
    &= (u_1 x \otimes u_2 + u_1 \otimes u_2 x) \otimes_{U(\p)} n.
\end{align*}
Here we used $x \in \Lie(H) \subseteq \p$. 
Now, combining these two terms and cancelling, we verify \eqref{eq:checkme}, and thus $\Ind(N)$ is a Harish-Chandra bimodule.

\begin{Corollary}\label{c:HCInd}
    There is an equivalence of categories 
    \begin{eqnarray}
    \label{eq:HCInd}
        \HC^{G^\chi}_{\chi} U(\g) \isoto \HC^{G^\chi}_{\chi} U(\g^{\chi_s})
    \end{eqnarray}
\end{Corollary}

\begin{proof}
    We have checked that \eqref{eq:consequentequivalence} restrict to well-defined functors between the subcategories appearing in \eqref{eq:HCInd}.
    Since both of them are full subcategories, these functors are adjoint, which one checks just like \eqref{eq:refme}. The unit and counit of the adjunction are isomorphisms, since this holds for the functors appearing in \eqref{eq:consequentequivalence}. This completes the proof.
\end{proof}

The equivalence \eqref{eq:HCInd} allows us to reduce to a category of Harish-Chandra bimodules (for a Levi subalgebra) such that the $p$-character restricts to a nilpotent $p$-character on the derived subalgebra. After making a similar reduction for $U_\chi(\g)\lmod$, Friedlander--Parshall reduced further to the case where the $p$-character is nilpotent \cite[Corollary~3.3]{FP}. They worked under slightly stronger hypotheses, however their argument was streamlined and extended to standard reductive groups (those satisfying (H1), (H2), (H3)) in \cite[B.9]{JaLATokyo}. We now extend this sequence of ideas to the setting of equivariant bimodules.

Since the composition $\t \to \g^{\chi_\rs} \to \g^{\chi_\rs}/ [\g^{\chi_\rs}, \g^{\chi_\rs}]$ is surjective, it follows that the abelian quotient $\g^{\chi_\rs}/ [\g^{\chi_\rs}, \g^{\chi_\rs}]$ has an invertible $p$-power map. Therefore we can choose $\lambda \in (\g^{\chi_\rs})^*$ such that $\lambda(x)^p - \lambda(x^{[p]}) = \chi_\rs(x)^p$ for all $x\in \g^{\chi_\rs}$. The corresponding $\g^{\chi_\rs}$-module $E$ has $p$-character $\chi_\rs$. We pick a generator and denote it $1_E$.

We equip $E$ with the trivial $G^\chi$-action, and we claim that $E \in U_{\chi_\rs}(\g^{\chi_\rs})\lmod^{G^\chi}$. The equivariance condition states that $\lambda$ is $G^\chi$-invariant.

Consider the set $\Lambda_{\chi_\rs} = \{ \lambda \in (\g^{\chi_\rs})^\ast \mid \lambda(x)^p - \lambda(x^{[p]}) = \chi_\rs(x)^p \text{ for all } x \in \g^{\chi_\rs} / [\g^{\chi_\rs},\g^{\chi_\rs}] \}$, a finite set of size $p^{\dim \g^{\chi_\rs} - \dim [\g^{\chi_\rs},\g^{\chi_\rs}]}$. Note that $G^{\chi_\rs}$ is connected and hence fixes $\Lambda_{\chi_\rs}$ pointwise. Therefore $G^\chi \subseteq G^{\chi_\rs}$ fixes $\lambda$, proving that $E$ is $G^\chi$-equivariant.

Since $E$ and its dual $E^*$ are $H$-equivariant, we have functors
\begin{equation} \label{eq: FP equivalence 3}
    U_\chi(\g^{\chi_s})\lmod^H_{\fg} \to U_{\chi_n}(\g^{\chi_s})\lmod^H_{\fg}, \qquad M \mapsto M \otimes E^*
\end{equation}
and
\begin{equation} \label{eq: FP equivalence 4}
    U_{\chi_n}(\g^{\chi_s})\lmod^H_{\fg} \to U_\chi(\g^{\chi_s})\lmod^H_{\fg}, \qquad V \mapsto V \otimes E.
\end{equation}
These are clearly quasi-inverse, and so we have proven.

\begin{Lemma} \label{thm: FP nilpotent reduction}
    The functors \eqref{eq: FP equivalence 3} and \eqref{eq: FP equivalence 4} give an equivalences of categories \begin{eqnarray}
    \label{eq: FP nilpotent reduction}
        U_\chi(\g^{\chi_s}) \lmod^{H}_{\fg}\isoto U_{\chi_n}(\g^{\chi_s}) \lmod^{H}_{\fg}
    \end{eqnarray}
\end{Lemma}

It is an easy exercise to upgrade this to Harish-Chandra bimodules, nonetheless we provide the details for the reader's convenience.

\begin{Corollary}\label{c:HCtwist}
    We have an equivalence of categories
    \begin{eqnarray}
    \label{eq:ogglepog}
        \HC^{G^\chi}_{\chi} U(\g^{\chi_s}) \isoto \HC^{G^\chi}_{\chi_n} U(\g^{\chi_s}).
    \end{eqnarray}
\end{Corollary}

\begin{proof}
    We define $E^{(2)} \in U_\chi(\g) \otimes U_{-\chi}(\g)\lmod^{G^\chi}$, letting $U_\chi(\g) \otimes 1$ act via $\lambda$ and $1\otimes U_{-\chi}(\g)$ act via $-\lambda$. Let $G^\chi$ act trivially on $E^{(2)}$.
    
    We need to verify that $M \otimes E \in \HC_\chi^{G^\chi} U(\g^\chi)$ for $M \in \HC_\chi^{G^\chi} U(\g^{\chi_\rs})$. It suffices to show that the strong equivariance condition is satisfied. For $m \in M$, $v \in E^{(2)}$, and $x \in \Lie(H)$ we need
    \begin{equation} \label{HC2 for mxe}
    (x \otimes 1 + 1 \otimes x) (m \otimes v) = d_1 \Ad(x) (m \otimes v).
    \end{equation}
    The left hand side of \eqref{HC2 for mxe} is $(\big(x\otimes 1 + 1\otimes x\big)\cdot m ) \otimes v + (\lambda(x) - \lambda(x)) m\otimes v$, which equals the right hand side because $G^\chi$ acts trivially on $E^{(2)}$.

    We have checked that $M \mapsto M \otimes E$ and $M \mapsto M\otimes E^*$ are functors between the categories \eqref{eq:ogglepog}. They are quasi-inverse because $E \otimes E^*$ is isomorphic to the trivial Harish-Chandra bimodule.
\end{proof}

    The following is the main theorem of the current section. Conceptually, the functor here is the quasi-inverse to \eqref{eq:HCequivalencewithequivariance+support} after appropriate specialisation whilst retaining equivariance.
\begin{Theorem} \label{thm: GF(X) equivariance}
    We have an equivalence
    \begin{eqnarray}
    \label{eq:thm GF(X) equivariance}
        U_\chi(\g) \lmod^{G^{F(\chi)}}_{\fg}\isoto\HC_\chi^{G^\chi} U(\g)
    \end{eqnarray}
\end{Theorem}

\begin{proof}
    Since $G^\chi$ is a closed subgroup scheme of $G$, and $G^{F(\chi)} = F^{-1}_G ((G^{\chi})^{(1)})$, we have an equivalence between $G^{F(\chi)} \lmod$ and $(U_0 (\g), G^\chi)\lmod$ by Lemma~\ref{L:monoidalequivalence}.
    Thus, the category $U_\chi(\g) \lmod^{G^{F(\chi)}}_{\fg}$ is equivalent to $U_\chi(\g) \lmod^{(U_0(\g), G^\chi)}_{\fg}$, using notation introduced in Section \ref{ss:equivariantmodules}.
    We will show that $U_\chi(\g) \lmod^{(U_0(\g), G^\chi)}_{\fg}$ is equivalent to the category $\HC_\chi^{G^\chi} U(\g)$.

    We recall the structures of modules in both categories.
    
    A module $M \in U_\chi(\g) \lmod^{(U_0(\g), G^\chi)}_{\fg}$ has three actions $\rho \colon U_\chi(\g) \rightarrow \End(M)$, $\sigma \colon U_0(\g) \rightarrow \End(M)$, and $\tau \colon G^\chi \rightarrow \GL(M)$. Both $\rho$ and $\sigma$ are $G^\chi$-equivariant, $\rho$ is $U_0(\g)$-equivariant and $d_1\tau = \sigma|_{\g^\chi}$.

    A module $M \in \HC_\chi^{G^\chi} U(\g)$ has three actions $\rho \colon U_\chi(\g) \rightarrow \End(M)$, $\upsilon \colon U_{\chi}(\g) \to \End(M)$ and $\tau : G^\chi \to \GL(M)$. Both $\rho$ and $\upsilon$ are $G^\chi$-equivariant commuting actions, and $d_1 \tau = (\rho - \upsilon)|_{\g^\chi}$.

    To construct quasi-inverse functors between these categories, we take $M \in U_\chi(\g) \lmod^{(U_0(\g), G^\chi)}_{\fg}$ and construct a Harish-Chandra bimodule with the same underlying vector space, the same actions $\rho, \tau$ and right action $\upsilon := \sigma - \rho$. The data $(\rho, \sigma, \tau)$ can be recovered from $(\rho, \upsilon, \tau)$, so it suffices to show that the functor is well-defined.
    
    We claim that $G^\chi$-equivariance of $\rho$ is equivalent to $\rho$ and $\upsilon$ commuting. Equivariance states that
    for $x, y \in \g$ and $g \in G^\chi$ we have
    \begin{align}
        \sigma(x) \rho(y) &= \rho(\ad(x)y) + \rho(y)\sigma(x). \label{eq: sigma equivariance}
    \end{align}
    whilst $[\rho, \upsilon] = 0$ states that
    \begin{equation}
    \label{eq:piggletog}
        [(\sigma - \rho)(x), \rho(y)] = [\sigma(x), \rho(y)] - [\rho(x), \rho(y)] = \sigma(x) \rho(y) - \rho(y) \sigma(x) - \rho(\ad(x)y) = 0.
    \end{equation}
    Equations \eqref{eq: sigma equivariance} and \eqref{eq:piggletog} are equivalent, and this completes the proof.
\end{proof}

The final corollary of this section will not be needed in the sequel, but we state and prove it here since it is quite surprising. In particular, note that the functor realising the following equivalence is not the forgetful functor, see Remark~\ref{R:curiousequiv}.
\begin{Corollary}
\label{C:semisimplesurpiseequivariance}
    If $\chi \in \g^*$ semisimple then we have an equivalence $U_\chi (\g) \lmod^{G^{F(\chi)}}_{\fg} \isoto U_\chi (\g) \lmod^{G^\chi}_{\fg}$.
\end{Corollary}

\begin{proof}
Theorem \ref{thm: GF(X) equivariance} gives the equivalence $U_\chi (\g) \lmod^{G^{F(\chi)}} \to \HC_\chi^{G^\chi} U(\g)$. Now applying Corollaries \ref{c:HCInd} and \ref{c:HCtwist} we obtain an equivalence with $\HC_0^{G^\chi} U(\g^\chi)$. Since $\g^{\chi_s} = \g^\chi$, we can apply Lemma~\ref{L:HCequivalencewithequivariance} to obtain an equivalence with $U_0 (\g^\chi) \lmod^{G^\chi}$. Finally Lemma~\ref{thm: FP equivalence} and Lemma~\ref{thm: FP nilpotent reduction} provide the equivalence with $U_\chi (\g) \lmod^{G^\chi}$.
\end{proof}

\begin{Remark}
    \label{R:curiousequiv}
    Note that the functor realising the equivalence in Corollary~\ref{C:semisimplesurpiseequivariance} is not the forgetful functor associated with the inclusion $G^\chi \subseteq G^{F(\chi)}$. Indeed, chasing through the proof we see that it is the functor of invariants with respect to $\theta(\u)$ where $\theta : \g \to \End(M)$ is the difference between the left action and the differential of the $G^{F(\chi)}$-action.
    
    The upshot is that $U_\chi(\g)\lmod^{G^\chi}_{\fg}$ encodes a `secret' restricted $\g$-action, which is far from obvious.
\end{Remark}

\subsection{Central characters for equivariant representations under parabolic induction}
\label{ss:centralcharactersunderparabolic}

Let $\O \subseteq \g^*$ be a semisimple orbit.

Combining \eqref{eq:HCequivalencewithequivariance+support} with Lemma~\ref{L:Zariskiclosedsemisimple}, Corollary~\ref{c:locofUXmod} and Theorem~\ref{thm: GF(X) equivariance} we have constructed an equivalence
\begin{eqnarray}
    \label{eq:specialisationfunctor}
    \HC_\O^G U(\g) \isoto \HC^{G^\chi}_\chi U(\g).
\end{eqnarray}
We claim that \eqref{eq:specialisationfunctor} is given by
\begin{eqnarray}
    M \longmapsto M / (I_\chi\otimes 1) M.
\end{eqnarray}
To see this, it suffices to observe that the functors \eqref{eq:HCequivalencewithequivariance+support} and \eqref{eq:thm GF(X) equivariance} are the identity functor on vector spaces, whilst the functor appearing in Corollary~\ref{c:kXmod=Hmodv2} is specialisation $M \mapsto M/(I_\chi\otimes 1) M$. It follows from Corollary~\ref{c:kXmod=Hmodv2} that the quasi-inverse is given by
\begin{eqnarray}
\label{eq:popcandy}
\begin{array}{rcl}
    \Ind_{G^{F(\chi)}}^G  \colon  \HC^{G^\chi}_\chi U(\g) & \to & \HC_\O^G U(\g)\\
    M & \mapsto & (\k[G] \otimes M)^{G^{F(\chi)}},
    \end{array}
\end{eqnarray}
where the $G^{F(\chi)}$-action is assembled from the $G^\chi$-action and the diagonal $U_0(\g)$-action, using Theorem~\ref{thm: GF(X) equivariance}.

Now, combining \eqref{eq:HCInd} with \eqref{eq:ogglepog} we obtain an equivalence
\begin{eqnarray}
\label{eq:HCrestrictedtoreduced}
\begin{array}{rcl}
    \HC_0^{G^\chi} U(\g^\chi) & \isoto & \HC_\chi^{G^\chi} U(\g),\\
    M & \longmapsto & U_{(\chi, -\chi)}(\g\oplus \g) \otimes_{U_{(\chi, -\chi)}(\p\oplus \p)} (M \otimes E^{(2)})
\end{array}
\end{eqnarray}
Recall from Corollary~\ref{c:HCtwist} that the module $E^{(2)}$ is determined by a choice of character $\lambda \in (\g^{\chi})^*$.

Suppose that $T \subseteq G^\chi$ is a maximal torus of $G$, and write $W^\chi$ for the Weyl group $N_{G^\chi}(T)/T$. Recall the Harish-Chandra centre of $U(\g^\chi)$ is $U(\g^\chi)^{G^\chi}$ and $\Spec U(\g^\chi)^{G^\chi} = \t^* / W^\chi_\bullet$. Write $$\theta^{(2)} : \t^*/W^\chi_\bullet \times \t^*/W^\chi_\bullet \to \t^* / W_\bullet \times \t^*/W_\bullet$$ for the quotient map.

The following result describes the effect of \eqref{eq:specialisationfunctor} and \eqref{eq:HCrestrictedtoreduced} on left and right Harish-Chandra central characters. The proof is straightforward and so we omit it.
\begin{Lemma} \
\label{L:cclemma}
    \begin{enumerate}
        \item The equivalence \eqref{eq:specialisationfunctor} respects left and right Harish-Chandra central characters.
        \item Translation by $\lambda$ on $\t^*$ descends to a well-defined morphism on $\t^* / W_\bullet^\chi$. If $M \in \HC_0^{G^\chi} U(\g^\chi)$ has (generalised) left and right HC central character
        $(\psi_1, \psi_2) \in \t^* / W^\chi_\bullet \times \t^* / W^\chi_\bullet$
        then after passing through \eqref{eq:HCrestrictedtoreduced} the (generalised) central character is
        $$\theta^{(2)}(\psi_1 + \lambda , \psi_2 - \lambda)  \ \in \ \t^* / W_\bullet \times \t^* / W_\bullet.$$
        $\hfill \qed$
    \end{enumerate}
\end{Lemma}

\section{Classification of simple modules}

From the previous sections, we may deduce that $\HC_\O^G U(\g)$ is equivalent to $\HC_0^{G^\chi} U(\g)$ when $\O$ is a semisimple coadjoint $G$-orbit and $\chi\in\O$. In this section we classify the simple modules   in $\HC_0^{G^\chi} U(\g)$ and thus deduce Theorem~\ref{T:MainThm}.

\subsection{Representations of the Frobenius neighbourhood of a reductive subgroup scheme}

Let $G$ be a standard reductive group, as in Section~\ref{ss:reductivegroupsandLiealgebras}. Let $H \subseteq G$ be a closed, connected subgroup scheme. Following Section~\ref{ss:stabilisersoftwistedcharacters} we consider the Frobenius neighbourhood $K := F^{-1}(H^{(1)})$ of $H$ in $G$. The goal of this section is to state and prove the classification of simple $K$-modules.

Fix a maximal torus and a Borel $T\subseteq B \subseteq G$, corresponding to simple roots $\Delta \subseteq \Phi \subseteq X(T)$. Recall the notation $X_1(T) = \{\lambda \in X(T) \mid 0 \le \langle \lambda ,\alpha^\vee\rangle < p \text{ for } \alpha \in \Delta\}$ from Section~\ref{ss:sketchclassificationintro}

\begin{Lemma}
\label{L:liftingtoG}
    If $M$ is a simple $G_1$-module then there is a simple $G$-module $\oM$ such that $\oM|_{G_1} = M$.
\end{Lemma}
\begin{proof}
    Thanks to \cite[Proposition~II.3.15(2)]{JanRAGS} the lemma will follow if we show that the composition $X_1(T) \hookrightarrow X(T) \onto X(T) / pX(T)$ is surjective.

    As per \cite[II.1.18]{JanRAGS} define $X_0(T) = \{\lambda \in X(T) \mid \langle \lambda, \alpha^\vee\rangle = 0$ for all $\alpha \in \Phi\}$. This coincides with the kernel of the restriction $X(T) \onto X(T)_{\operatorname{der}}$ to the group of characters of $T \cap G_{\operatorname{der}}$, where $G_{\operatorname{der}}$ denotes the derived subgroup. Since $X(T)_{\operatorname{der}}$ is torsion free, we can write $X(T) = X_0(T) \oplus M$ for some free abelian subgroup $M$. Clearly the composition $M \to X(T) \to X(T)_{\operatorname{der}}$ is an isomorphism.

    Now we can make identifications
    \begin{eqnarray*}
        X(T)/pX(T) & = & X_0(T) / pX_0(T) \oplus M / p M,\\
        X_1(T) & = & X_0(T) \oplus (M \cap X_1(T))
    \end{eqnarray*}
    Since $G_{\operatorname{der}}$ is assumed to be simply connected, \cite[Remark~II.3.15(2)]{JanRAGS} implies that $M \cap X_1(T) \onto M / pM$, and this implies $X_1(T) \onto X(T) / pX(T)$. This completes the proof.
\end{proof}

If $M$ is a simple $K$-module then $M^{(1)}$ is a simple $F_G(K) = H^{(1)}$-module. We write $M^F$ for the pullback of $M^{(1)}$ to $K\lmod$ through $F_G$.

If $M$ is a simple $G_1$-module then $\oM|_K$ is a simple $K$-module which restricts to $M$ on $G_1$. The following result is a generalisation of \cite[Theorem~5.5]{Ho}, and both follow a similar argument to \cite[II.3.16(1)]{JanRAGS}.
\begin{Theorem}
\label{T:simplesFrobeniusneighbourhood}
    Let $\Lambda_1, \Lambda_2$ be sets such that $\{M_\lambda \in H\lmod \mid \lambda \in \Lambda_1\}$ are representatives of the isomorphism classes of simple $H$-modules, and $\{N_\lambda \in G_1\lmod \mid \lambda \in \Lambda_2\}$ are representatives of the isomorphism classes of simple $G_1$-modules.
        Then the modules
        \begin{eqnarray}
            \label{eq:classifysimplestensor}
            \{(\circleover{N}_{\lambda_1})|_K \otimes M_{\lambda_2}^F \mid \lambda_1 \in \Lambda_1, \ \lambda_2 \in \Lambda_2\}
        \end{eqnarray}
        form a complete set of non-isomorphic simple $K$-modules.
\end{Theorem}
\begin{proof}
    {\it Step 1:} First of all, we show that every simple $K$-module has the form \eqref{eq:classifysimplestensor}.
    
    Note that $G_1 \subseteq K$ is a normal subgroup, that $H \subseteq K$ is a group subscheme and (by Lemma~\ref{L:semidir}) that $K$ is generated by $G_1$ and $H$, i.e. $K = G_1 H$.
    
    Suppose that $V$ is a simple $K$-module. Since $\k$ is perfect and $H$ is connected, we can apply \cite[Corollary~18.3]{BorelLAGS} to see that $H(\k)$ is dense in $H$. As a consequence, we can apply the remark after \cite[Proposition~I.6.16]{JanRAGS} to deduce that the socle $\soc_{G_1}(V)$ is a $K$-submodule. By \cite[I.2.14(2)]{JanRAGS} we see that $\soc_{G_1}(V) \ne 0$, and so $\soc_{G_1}(V) = V$. Thus $V|_{G_1}$ is semisimple.

    If $L$ is a simple $G_1$-module then we can equip $\Hom_{G_1}(\circleover{L}|_K, V)$ with a $K$-module structure, with $G_1$ acting trivially. It is not hard to see that the following is a $K$-equivariant map
    \begin{eqnarray*}
    \Hom_{G_1}(\circleover{L}|_K, V) \otimes \circleover{L}|_K &\longrightarrow & V \\
        \phi \otimes v & \longmapsto & \phi(v)
    \end{eqnarray*}
    Consider the direct sum of such maps
     \begin{eqnarray}
     \label{eq:directsumiso}
    \bigoplus_{[L] \in \operatorname{Irr}(G_1)}\Hom_{G_1}(\circleover{L}|_K, V) \otimes \circleover{L}|_K &\longrightarrow & V
    \end{eqnarray}
    where the direct sum varies over the isomorphism classes of simple $G_1$-modules. This is $K$-equivariant, and it is an isomorphism of $G_1$-modules because $V|_{G_1}$ is semisimple. Hence it is an isomorphism of $K$-modules.

    If $L$ is a simple $G_1$-module then the image of $\Hom_{G_1}(\circleover{L}|_K, V) \otimes \circleover{L}|_K$ under \eqref{eq:directsumiso} is a $K$-submodule of $V$; it is nonzero if $L$ appears as a composition factor of $V|_{G_1}$. Since $V$ is simple, there is a unique such isomorphism class $L$. This shows that $V|_{G_1}$ is semisimple and isotypic.

    Now let $N$ denote the unique simple $G_1$-module appearing in $V|_{G_1}$. Write $M := \circleover{N}|_K$ and $M := \Hom_{G_1}(N, V)$. The image of $N \otimes M \to V$ is nonzero hence the map is surjective. Considering the restriction to $G_1$, it is also injective.  Noting that $M|_{G_1}$ has trivial $G_1$-action, we have shown that $V \cong N \otimes M$ has the required form.

    {\it Step 2:} Next we show that if $M, N$ are simple $K$-modules, with $M|_{G_1}$ trivial and $N|_{G_1}$ simple, then $N \otimes M$ is simple. If not, then there is a nonzero proper simple submodule $S \subseteq N \otimes M$.

    Observe that $m \mapsto (n\mapsto n\otimes m)$ gives an injective $K$-linear map
    \begin{eqnarray}
    \label{eq:Hisotosomething}
        M \longrightarrow \Hom_{G_1}(N, N\otimes M).
    \end{eqnarray}
    Since $(N\otimes M)|_{G_1}$ is semisimple and isotypic, we have an isomorphism 
    \begin{eqnarray}
    \label{eq:dimcomp}
        \Hom_{G_1}(N, N\otimes M) \otimes N \longrightarrow N\otimes M,
    \end{eqnarray}
    reasoning as per the second and third paragraphs of Step 1 of the current proof. Now, a dimension comparison using \eqref{eq:dimcomp} confirms that $\dim M = \dim \Hom_{G_1}(N, N\otimes M)$, which implies that \eqref{eq:Hisotosomething} is an isomorphism.
    
    Using the left-exactness of $\Hom_{G_1}(N, -)$ we see that $\Hom_{G_1}(N, S) \subseteq \Hom_{G_1}(N, N\otimes M)$ is a proper nonzero $K$-submodule -- the dimension is the composition multiplicity of $N$ in $S|_{G_1}$, which is strictly less than the composition multiplicty of $N$ in $(N \otimes M)|_{G_1}$. We conclude that $M$ is not simple, which contradicts our assumptions. Hence $N \otimes M$ is simple, as claimed.

     {\it Step 3:} Finally we argue that the modules \eqref{eq:classifysimplestensor} are pairwise non-isomorphic. Indeed, if $$\circleover{N}_{\lambda_1}|_K \otimes M_{\lambda_2}^F \cong \circleover{N}_{\lambda_1'}|_K \otimes M_{\lambda_2'}^F$$ for indexes $\lambda_1, \lambda_1', \lambda_2, \lambda_2'$ then by restricting both to $G_1$ we see that $\lambda_1 = \lambda_1'$, using Step 1. Furthermore, Step 2 of the proof implies that $$M_{\lambda_2}|_H \cong \Hom_{G_1}(\circleover{N}_{\lambda_1}|_H, N_{\lambda_1} \otimes M_{\lambda_2}) \cong \Hom_{G_1}(\circleover{N}_{\lambda_1'}|_H, N_{\lambda_1'} \otimes M_{\lambda_2'}) \cong M_{\lambda_2'}|_H.$$
     This completes the proof.
\end{proof}

\subsection{Simple restricted Harish-Chandra bimodules}

Let $G$ be a standard reductive group. Now we apply Theorem~\ref{T:simplesFrobeniusneighbourhood} to classify the simple objects in $\HC_0^G U(\g)$, completing the proof of Theorem~\ref{T:MainThm}.

Fix a choice of maximal torus and Borel subgroup $T \subseteq B \subseteq G$ and write $\t \subseteq \b \subseteq \g$ for their Lie algebras. Let $h_1,...,h_r \in \t$ denote a basis consisting of toral elements of $\t$, i.e. $h_i^{[p]} = h_i$ for $i=1,...,r$.

Write
\begin{eqnarray*}
\Lambda_0 &:=& \{ \lambda \in \t^* \mid \lambda(h_i) \in \F_p \text{ for all } i = 1,...,r\} \\
& = & \{\lambda\in \t^* \mid \lambda(h)^p = \lambda(h^{[p]}) \text{ for all } h\in \t\}.
\end{eqnarray*}

Taking the differential at $1\in \G_m^\times$ gives a homomorphism
\begin{eqnarray}
    d_1 : X(T) \to \Lambda_0
\end{eqnarray}
This map has kernel $p X(T)$ and induces an isomorphism $X(T) / pX(T) \isoto \Lambda_0$.

For $\lambda \in \Lambda_0$ we let $\k_\lambda$ denote the one dimensional $B$-module associated to $\lambda$, and by slight abuse of notation we use the same notation for the one dimensional (restricted) $\b$-module determined by $d_1\lambda$.

Following \cite[II.2.1]{JanRAGS} we let $H^0(\lambda) := R^0 \Ind_B^G(\k_\lambda)$, the universal highest weight $G$-module generated by a weight vector of weight $\lambda$. According to Corollary~II.2.3 and Corollary~II.2.7 of \cite{JanRAGS} the socle $L(\lambda) := \soc_G H^0(\lambda)$ is simple and the modules
\begin{eqnarray}
    \{L(\lambda) \mid \lambda \in X(T)_+\}
\end{eqnarray}
form a complete list of non-isomorphic simple $G$-modules.

Similarly, for $\lambda \in X(T)$ we can define the corresponding baby Verma module $Z_0(\lambda) := U_0(\g) \otimes_{U_0(\b)} \k_\lambda$. According to \cite[\textsection 10]{JaLA} every $Z_0(\lambda)$ admits a unique simple quotient $Z_0(\lambda) \onto L_0(\lambda)$. Furthermore, we have $L_0(\lambda) \cong L_0(\mu)$ if and only if $\lambda -\mu \in p X(T)$, and the modules
\begin{eqnarray}
\{L_0(\lambda) \mid \lambda \in X(T) / pX(T)\}
\end{eqnarray}
form a complete list of non-isomorphic $U_0(\g)$-modules. It is a classical theorem of Curtis \cite{Cu} states that for $\lambda \in X_1(T)$ the differential of the representation $G \to \End L(\lambda)$ is isomorphic to $\g \to \End L_0(\lambda)$.

Consider the outer tensor product
$$- \boxtimes - : U_0(\g)\lmod \times U_0(\g) \lmod \to U_0(\g) \otimes U_0(\g) \lmod \cong U_0(\g \oplus \g)\lmod.$$
Thanks to \eqref{eq:Frobeniuskernelandrestrictedenvelopingalgebra} if $\lambda, \mu \in X(T) / p(T)$ then we can regard $L_0(\lambda) \boxtimes L_0(\mu)$ as a simple $G_1 \times G_1$-module, and by Lemma~\ref{L:liftingtoG} it can be regarded as a $G\times G$-modules, hence equipped with the diagonal action via $\Delta : G\to G\times G$. The $U_0(\g) \otimes U_0(\g)$-module and $\Delta(G)$-module structures make $L_0(\lambda) \boxtimes L_0(\mu)$ into an object in $\HC_0^G U(\g)$.

On the other hand, if $F : G \times G \to G^{(1)} \times G^{(1)}$ denotes the Frobenius morphism and $M \in \Delta(G)\lmod$ then we write $M^F \in U_0(\g\oplus \g)\lmod^G$ for the pullback of $M^{(1)}$ through $F$, equipped with the trivial $U_0(\g\oplus \g)$-module structure.

\begin{Theorem}\label{T:classrestsimples}
    A complete set of non-isomorphic simple objects in $\HC_0^G U(\g)$ is given by
    \begin{eqnarray}
        \{(L(\lambda) \boxtimes L(\mu)) \otimes L(\gamma)^F \mid \lambda, \mu \in \Lambda_0, \ \gamma \in X(T)_+\}.
    \end{eqnarray}
\end{Theorem}
\begin{proof}
    Let $K := F_G^{-1}(\Delta(G)^{(1)}) \subseteq G \times G$, the Frobenius neighbourhood of the diagonal subgroup. The category of finitely generated, strongly equivariant modules $(U_0(\g \oplus \g), \Delta(G))\lmod_{\fg}$ is nothing but $\HC_0^G U(\g)$, and so the current theorem follows directly from Theorem~\ref{T:simplesFrobeniusneighbourhood} using the monoidal equivalence proven in Lemma~\ref{L:monoidalequivalence}. Note that the latter induces an equivalence between finite-dimensional modules and finitely generated modules.
\end{proof}

\subsection{Classification of simple Harish-Chandra bimodules}

In this final section we put together the results of the previous subsections to establish Theorem~\ref{T:MainThm}.

Let $G$ be a connected reductive algebraic group scheme over $\k$, and assume the standard hypotheses (H1), (H2), (H3) from Section~\ref{ss:reductivegroupsandLiealgebras}.

If $M$ is a simple module in $\HC^G U(\g)$ then Lemma~\ref{L:suportsemisimple} implies that $\Supp_p(M)$ is a semisimple $G$-orbit in $(\g^*)^{(1)}$. Let $\O$ be a semisimple $G$-orbit in $\g^*$ such that $\Supp_p(M)=\O^{(1)}$. Writing $\HC_\O^G U(\g)$ in place of $\HC_{\O^{(1)}}^G U(\g)$ for ease of notation, we thus obtain that each simple module in $\HC^G U(\g)$ is a simple module in $\HC^G_{\O} U(\g)$ for some unique semisimple $G$-orbit $\O$ (and simple modules in $\HC^G_{\O} U(\g)$ clearly remain simple in $\HC^G U(\g)$). Thus, to classify all simple modules in $\HC^G U(\g)$ it suffices to classify those in $\HC^G_\O U(\g)$ as $\O$ runs over the semisimple $G$-orbits in $\g^*$. 

For the remainder of this subsection, fix a semisimple $G$-orbit $\O$ in $\g^*$ and fix $\chi\in \g^*$. From the previous subsections, we obtain the following chain of equivalences of categories.
$$\HC^G_\O U(\g)\xleftrightarrow{\eqref{eq:HCequivalencewithequivariance+support}} U_\O(\g)\lmod_{\fg}^{G} \xleftrightarrow{\tiny\mbox{Cor. } \ref{c:locofUXmod} } U_\chi(\g)\lmod_{\fg}^{G^{F(\chi)}}\xleftrightarrow{\tiny\mbox{Thm.}~\ref{thm: GF(X) equivariance}} \HC_\chi^{G^\chi}U(\g)\xleftrightarrow{\tiny\mbox{Cor. }\ref{c:HCInd} \ \&\ \ref{c:HCtwist}} \HC_0^{G^\chi} U(\g^\chi).$$ To classify the simple modules in $\HC^G_\O U(\g)$ it thus suffices to classify the simple modules in $\HC_0^{G^\chi} U(\g^\chi)$.

Note that $G^\chi$ is a Levi subgroup of $G$ and thus satisfies the standard hypotheses. Theorem~\ref{T:MainThm} then follows from Theorem~\ref{T:classrestsimples} applied to $G^\chi$.

The final assertion on central characters follows from Lemma~\ref{L:cclemma}.

\end{document}